\documentclass[12pt]{amsart}
\usepackage[T1]{fontenc}
\usepackage[all]{xy}
\usepackage{amssymb,amsmath,xypic,ifthen}
\usepackage{graphicx,soul,color}
\newcommand{\marginnote}[1]{\ifthenelse{\isodd{\thepage}}{\normalmarginpar}
{\reversemarginpar}\marginpar{\fbox{\parbox{15mm}{\tt\sloppy\footnotesize #1}}}}

\newcommand{\qee}{ \hfill\hspace{2pt}$\triangle$}
\newtheorem{thm}{Theorem}[section]
\newtheorem{corol}[thm]{Corollary}
\newtheorem{lemma}[thm]{Lemma}
\newtheorem{prop}[thm]{Proposition}
\newtheorem{defin}[thm]{Definition}

\theoremstyle{remark}
\newtheorem{rema}[thm]{Remark}
 
\newtheorem{exe}[thm]{Example}
 
\newcommand{\alp}{{\alpha}}
\newcommand{\cO}{{\mathcal O}}

\newcommand{\PP}{{\mathbb P}}
\newcommand{\Ext}{{\rm Ext}}
\newcommand{\Hom}{{\rm Hom}}
\newcommand{\pr}{{\operatorname{pr}}}

\newcommand{\LL}{{\mathcal L}}
\newcommand{\LLL}{{\mathcal L}}

\newcommand{\F}{{\mathcal F}}
\newcommand{\FFF}{{\mathcal F}}
\newcommand{\E}{{\mathcal E}}
\newcommand{\EEE}{{\mathcal E}}

\newcommand{\G}{{\mathcal G}}

\newcommand{\K}{{\mathcal K}}
\newcommand{\Q}{{\mathcal Q}}
\newcommand{\QQ}{{\mathbb Q}}

\newcommand{\HH}{{\mathcal H}}

\newcommand{\C}{{\mathbb C}}

\newcommand{\ch}{\operatorname{ch}}
\newcommand{\id}{\operatorname{id}}
\newcommand{\im}{\operatorname{im}}
\renewcommand{\max}{{\operatorname{max}}}

\newcommand{\rk}{\operatorname{rk}}

\newcommand{\Quot}{{\operatorname{Quot}}}
\newcommand{\Spec}{{\operatorname{Spec}}}
\newcommand{\Tor}{{\operatorname{Tor}}}

\newcommand{\quot}{{\operatorname{Quot}(\HH,P_c)}}
\newcommand{\fquot}{{\operatorname{{\mathbb Q}uot}(\HH,P_c,\F)}}

\newcommand{\rar}{\rightarrow}
\newcommand{\dual}{{\scriptscriptstyle \vee}}

\newlength{\rrrr}

\newcommand{\isoto}{{\longrightarrow\hspace{-1.5 em}
\raisebox{ 0.6 ex}{$\textstyle\sim$}\hspace{0.8 em}}}

\definecolor{orange}{rgb}{1.0, 0.5, 0.0}
\definecolor{monbleu}{rgb}{0.25, 0.25, 0.75}

\begin{document}
\bigskip\bigskip
\title[Uhlenbeck-Donaldson compactification for framed sheaves]{Uhlenbeck-Donaldson compactification for \\[5pt] framed sheaves on projective surfaces}
\bigskip
\date{Revised version March 31st, 2013}
\thanks{This research was partly supported  by   {\sc prin}    ``Geometria delle variet\`a algebriche e dei loro spazi di moduli'', by Istituto Nazionale di Alta Matematica and by the grant VHSMOD-2009 No.  ANR-09-BLAN-0104-01.}
 \maketitle \thispagestyle{empty}
\begin{center}
{\sc Ugo Bruzzo}$^\ddag$, {\sc Dimitri Markushevich}$^\S$
and {\sc Alexander Tikhomirov}$^\P$
\\[10pt]  {\small
$\ddag$ Scuola Internazionale Superiore di Studi Avanzati, \\ Via Bonomea 265, 34136
Trieste, Italia \\ and Istituto Nazionale di Fisica Nucleare, Sezione di Trieste \\
{\tt bruzzo@sissa.it}
}
\\[10pt]  {\small
\S Math\'ematiques --- B\^at. M2, Universit\'e Lille 1, \\ F-59655 Villeneuve d'Ascq Cedex, France \\
{\tt  markushe@math.univ-lille1.fr} }
\\[10pt]  {\small
\P Department of Mathematics, Yaroslavl State Pedagogical University,\\  Respublikanskaya Str. 108,
150 000 Yaroslavl, Russia \\  {\tt astikhomirov@mail.ru}}
\end{center}

\bigskip\bigskip

\begin{abstract}
We construct a compactification  $M^{\mu ss}$ of the Uhlenbeck-Donaldson type
for the moduli space of slope stable framed bundles. This is
a kind of a moduli space of slope semistable framed sheaves. We show that there exists
a projective morphism $\gamma \colon M^{ss} \to M^{\mu ss}$, where
$M^{ss}$ is the moduli space of S-equivalence classes of Gieseker-semistable
framed sheaves. The space  $M^{\mu ss}$  has a natural set-theoretic stratification
which allows one, via a Hitchin-Kobayashi correspondence,  to compare it with the moduli spaces of framed ideal instantons.
  \end{abstract}
 \maketitle
 \markright{\sc U. BRUZZO, D. MARKUSHEVICH and A. TIKHOMIROV}
   \section{Introduction}
Let $X$ be a smooth complex projective surface, and let $M^\mu(c)$ be the moduli space of $\mu$-stable (also called
slope stable) locally-free coherent $\cO_X$-modules of numerical class $c$.
One can obtain a compactification
of $M^\mu(c)$ by taking closure in the Gieseker-Maruyama moduli
space space $M^{ss}(c)$, formed
by the S-equivalence classes of Gieseker-semistable coherent $\cO_X$-modules.
On the other hand, by the so-called
Hitchin-Kobayashi correspondence \cite{LuTel}, $M^\mu(c)$ may  be regarded
as a moduli space of bundles carrying a Hermitian-Yang-Mills metric;
as such, it admits a differential-geometric compactification,
called the Uhlenbeck-Donaldson compactification $N(c)$, which is obtained
by adding to $M^\mu(c)$ two types of degenerate objects: singular Hermitian-Yang-Mills metrics, called ``ideal instantons'', and reducible metrics, called ``parabolic ends''.
In a 1993 paper Jun Li showed that $N(c)$ may be given a structure of a scheme
over $\C$, and constructed a morphism $M^{ss}(c)\to N(c)$, which on $M^\mu(c)$
restricts to an isomorphism  \cite{JunLi}. With that scheme structure, $N(c)$ may be regarded
as a sort of moduli space of $\mu$-semistable sheaves, under an identification which is
somehow stronger than S-equivalence \cite{HLbook}.

In this paper we consider pairs formed by a bundle  on a smooth polarized projective surface, together with a framing.
A notion of stability exists for such objects, depending on an additional parameter $\delta$ which is a polynomial
with rational coefficients. This notion of stability gives rise to a GIT moduli space \cite{HL1,HL2}.
The main result of this paper is the construction of an Uhlenbeck-Donaldson compactification
$M^{\mu ss}(c,\delta)$ for the $\mu$-polystable part
$M^{\mu\mbox{\scriptsize -}\mathrm{poly}}(c,\delta)$
of this moduli space. This is accomplished by following rather closely the
construction of the Uhlenbeck-Donaldson compactification of
the moduli space of (unframed) vector bundles, as done, e.g., in \cite{HLbook}.
A first key ingredient is, as always, a boundedness result for the family
$\mathcal{S}^{\mu ss}(c,\delta)$
of $\mu$-semistable framed sheaves on $X$ (Proposition \ref{bounded}) with numerical class $c$.
After introducing an appropriate Quot scheme, this family is realized as a locally closed
subset $R^{\mu ss}(c,\delta)$ in the Quot scheme, and a suitable semiample line bundle on $R^{\mu ss}(c,\delta)$
is picked out. The moduli scheme $M^{\mu ss}(c,\delta)$ cannot be defined as a geometric quotient, hence it is defined
in an {\em ad hoc} way, cf.~Definition  \ref{def M-muss}. The Jordan-H\"older filtration allows one to introduce a
set-theoretic stratification in the space $M^{\mu ss}(c,\delta)$.

Let $X$ be a smooth projective surface,
$D$ a divisor  on $X$ satisfying some numerical conditions, and $\F$ a rank $r$ vector bundle
on $D$, which is semistable or satisfies a slightly more general
stability condition.
The following property was proved in \cite{DimaUgo}: given a
torsion-free rank $r$ sheaf $\E$  on $X$ and an isomorphism
$\phi\colon \E{\vert D}\to\F$, one
can choose a polarization $H$ in $X$ and a stability condition for framed
sheaves in such a way that the pair  $(\E,\phi)$
is stable in Huybrechts-Lehn's sense.
Moreover,
the choice of the polarization and that of the
stability condition only depend on the pair $(D,\F)$ and on the numerical class
of $\E$. This means that the moduli space of such pairs
embeds into a moduli space of stable pairs, and therefore we can restrict
the Uhlenbeck-Donaldson compactification to it. Via the Hitchin-Kobayashi correspondence, this allows
one to look at $M^{\mu ss}(c,\delta)$ as a {quasi-projective scheme structure
on the} compactified
moduli space of instantons.
{In the framed case,
a quasi-projective Uhlenbeck-Donaldson type compactification
has been previously known only
for $X=\PP^2$. It was constructed by Nakajima in} {\cite{Nak0}}
{by completely different techniques, using the ADHM
data and hyperk\"ahler quotients.}

\smallskip {\bf Acknowledgments.} This paper was mostly written during the first and third author's stays at Universit\'e de Lille I. They acknowledge the financial support and the warm hospitality. The authors also thank D.~E.~Diaconescu, D.~Huybrechts and F.~Sala for discussions, and the anonymous referee for useful remarks and suggestions.

 \bigskip

    \section{A Quot scheme for framed sheaves}\label{Quot}
  Let $X$ be a smooth $d$-dimensional projective variety  over an algebraically closed field $\Bbbk$ of characteristic zero, $H$ an ample class on it, $\F$ a
coherent sheaf on $X$, $c\in K(X)_{\mbox{\tiny num}}$
a numerical K-theory class, $P_c$ the corresponding Hilbert polynomial.
We  shall consider pairs $(\E,[\phi])$,  where $\E$ is a coherent sheaf on $X$
with Hilbert polynomial $P_\E=P_c$, and {$[\phi]\in \PP(\Hom(\E,\F)^*)$
is the proportionality class of nonzero sheaf morphism $\phi\colon\E\to\F$}.
We call each such pair $(\E,[\phi])$ a {\it framed sheaf}. Later on, to simplify notation, we shall write a framed sheaf as $(\E,\phi)$. {A homomorphism between two framed sheaves
$(\E_1,[\phi_1])$, $(\E_2,[\phi_2])$ is a sheaf homomorphism $f:\E_1\rar\E_2$
such that $\phi_2f=\lambda \phi_1$ for some $\lambda\in k$. An isomorphism
is an invertible homomorphism.}
At some stage, when we consider (semi)stability of framed sheaves, also the choice of a
polynomial $\delta$ will come into play.

Let $V$ be a vector space of dimension $P_c(m)$ for some $m\gg 0$,
let $\HH=V\otimes\cO_X(-m)$, and let $\quot$ be the Quot scheme parametrizing
the coherent quotients of $\HH$ with Hilbert polynomial $P_c$. On
$\quot\times X$ there is a universal quotient $\tilde\Q$, and a morphism
$$\cO_{\quot} \boxtimes\HH  \stackrel{\tilde g} {\to}\tilde\Q$$
Let $\PP=\PP[\Hom(V,H^0(X,\F(m)))]^\ast$;  the points of $\PP$ are in a one-to-to correspondence
with morphisms $ \HH\to\mathcal{F}$ up to
a constant factor.
Let $Y:=\fquot$ be the closed subscheme of $\quot\times\PP$ formed
by the pairs $([g],[a])$ such that there is {a} morphism $\phi\colon \G\to\F$
for which the diagram
$$\xymatrix{\HH \ar[r]^g\ar[dr]_a & \G \ar[d]^\phi \\ & \F}$$
commutes.
{Obviously, such $\phi$ is uniquely determined by $a$.}
We denote by $\Q$ the restriction to $Y\times X$ of the pullback
of $\tilde\Q$ to $\quot\times\PP\times X$.
There is
a line bundle $\LL_Y$ on $Y$ and
a morphism
$\Phi\colon\Q\otimes p^*\LL_Y\to q^*\F$,
where $Y\overset{p}\leftarrow Y\times X\overset{q}\to X$ are the natural projections.

For a given scheme $S$ let $S\xleftarrow{\pr_1} S\times X\xrightarrow{\pr_2} X$ be the projections.
\begin{defin}\label{relfram}
A family $\mathbf{F}=(\mathbf{E},\mathbf{L},\alpha_{\mathbf{E}})$
(or, shortly, $\mathbf{F}=(\mathbf{E},\alpha_\mathbf{E})$) of framed sheaves on $X$ parametrized by a scheme $S$
is a  sheaf\ \ $\mathbf{E}$ on $S\times X$, flat over $S$,
a line bundle $\mathbf{L}$ on $S$, and a subbundle morphism
$\alpha_{\mathbf{E}}:\mathbf{L}\to \mathrm{pr}_{1*}\mathcal{H}om(\mathbf{E},\mathcal{O}_S\boxtimes\mathcal{F})$
called a framing of $\mathbf{E}$.
Two families
$(\mathbf{E},\mathbf{L},\alpha_\mathbf{E})$ and $(\mathbf{E}',\mathbf{L}',\alpha_{\mathbf{E}'})$
are called isomorphic if there exist isomorphisms
$g:\mathbf{E}\isoto\mathbf{E}'$ and  $h:\mathbf{L}\isoto\mathbf{L}'$
such that $\alpha_{\mathbf{E}}\cdot\tilde{g}=\alpha_{\mathbf{E}'}\cdot h$, where
$\tilde{g}:\mathrm{pr}_{1*}\mathcal{H}om(\mathbf{E},\mathcal{O}_S\boxtimes\mathcal{F})\to
\mathrm{pr}_{1*}\mathcal{H}om(\mathbf{E}',\mathcal{O}_S\boxtimes\mathcal{F})$
is the isomorphism induced by $g$.
\end{defin}
We may look at a framing
$\alpha_{\mathbf{E}}:\mathbf{L}\to \mathrm{pr}_{1*}\mathcal{H}om(\mathbf{E},\mathcal{O}_S\boxtimes\mathcal{F})$
as a nowhere vanishing morphism
$\widetilde{\alpha}_{\mathbf{E}}:\mathrm{pr}_1^*\mathbf{L}\otimes\mathbf{E}\to\mathcal{O}_S\boxtimes\mathcal{F}$,
defined as the composition
$\mathrm{pr}_1^*\mathbf{L}\otimes\mathbf{E}\xrightarrow{\mathrm{pr}_1^*\alpha_{\mathbf{E}}}
\mathrm{pr}_1^*\mathrm{pr}_{1*}\mathcal{H}om(\mathbf{E},\mathcal{O}_S\boxtimes\mathcal{F})\otimes\mathbf{E}
\overset{ev}\to\mathcal{H}om(\mathbf{E},\mathcal{O}_S\boxtimes\mathcal{F})\otimes\mathbf{E}\xrightarrow{can}
\mathcal{O}_S\boxtimes\mathcal{F}$.



\begin{defin}
Let $(\E,[\phi])$ be a framed sheaf on $X$.
A pair $(\G,[\psi])$ is a quotient of $(\E,[\phi])$ if $\G$ is a quotient
of $\E$, and the diagram
$$\xymatrix{ \E \ar[r]\ar[dr]_\phi & \G\ar[d]^\psi \\ &\F}$$
commutes modulo a scalar factor.

 If $(\E,[\phi])$ is a
framed sheaf on $X$, a family of framed quotients of $(\E,[\phi])$ is a family
of framed sheaves $(\G,\LL,\Psi)$ over a scheme $S$ with a sheaf epimorphism
$g\colon \operatorname{pr}_2^\ast\E \to \G$,
a line bundle $\LL$ on $S$
and a section $s\in \Gamma(S, \LL^\dual)$, such that the diagram
$$\xymatrix{  \mathrm{pr}_1^*\LL\otimes\operatorname{pr}_2^\ast\E \ar[rrrr]^{\mathrm{pr}_1^*\LL\otimes g}
\ar[drr]_{ \operatorname{pr}_1^\ast s\,\otimes\,\id}&&&& \mathrm{pr}_1^*\LL\otimes \G\ar[dd]^{\Psi} \\
 &&\operatorname{pr}_2^\ast\E \ar[drr]_{ \operatorname{pr}_2^\ast\phi}  && \\
&&&&\operatorname{pr}_2^\ast\F}$$
commutes.
\end{defin}

The universality property of the Quot scheme implies the following result.

\begin{prop} \label{univprop}Let $(\G,\Psi)$ be a   family of  framed quotients of
$\HH$, parametrized by a scheme $S$.
Assume that the Hilbert polynomial of $\G_s=:\G\otimes\Bbbk(s)$ is $P_c$
for any $s\in S$. Then there is a morphism $f\colon S \to \fquot$  (unique up
to a unique  isomorphism) such that $(\G,\Psi)$ is
isomorphic to $(f\times \operatorname{id})^\ast(\Q,\Phi)$ over $S$.
\end{prop}

{{The action of $\operatorname{SL}(V)$ on $V$ induces well-defined actions on $\quot$ and $\PP$ which
are compatible, so that one has an action of $\operatorname{SL}(V)$ on $Y=\fquot$}}. The moduli space of
semistable framed sheaves is constructed as the GIT quotient of $Y$ by this action
of~$\operatorname{SL}(V)$~\cite{HL1}.

For the reader's convenience, and basically following \cite[Ch.~8]{HLbook},
we briefly recall the construction of determinant line bundles on the Quot scheme.
For a scheme $Z$, we shall denote by $K(Z)$ and $K^0(Z)$  the Grothendieck
groups of coherent    and of  locally-free $\mathcal O_Z$-modules respectively.
Let $X$ be a smooth projective variety,  and let $\E$ be a flat family of coherent
sheaves on $X$ parametrized by a scheme $S$. Note that $\E$ singles out a well defined
class $[\E]$ in  $K^0(S\times X)$. If $p$ and $q$ are the projections onto the
two factors of $S\times X$, one defines a morphism $\lambda_{\E}\colon K(X)\to\operatorname{Pic}(S)$
by letting
\begin{equation}\label{det1}
\lambda_\E(u) = \operatorname{det} p_! \,(q^\ast (u) \cdot [ \E] )
\end{equation}
where $p_!$ is the (well defined) morphism $K^0(S\times X)\to K^0(S)$ induced by $p$. Later on
we shall use the line bundle $\lambda_{\widetilde{\Q}}(u_1)$ on $\quot$, where $\widetilde{\Q}$ is the universal
quotient on $\quot\times X$, and
\begin{equation}\label{det2}
u_i = u_i(c) = - r \cdot h^i + \chi (c\cdot h^i)\cdot [\mathcal O_x]\,.
\end{equation}
Here $r=\operatorname{rk}(\widetilde{\Q})$, $h$ is the class of $\mathcal O_X(1)$ in $K(X)$,
and $x$ is a fixed point of $X$.

\bigskip
\section{A family  of $\mu$-semistable framed sheaves on a surface} \label{family}
From now on we are assuming that the framing sheaf is supported on a divisor. Let $(X,H)$ be
a smooth projective variety of dimension $d\geq 1$ with an ample divisor $H$, $D\subset X$ an effective divisor, and $\mathcal{F}$ an $\mathcal{O}_D$-module of dimension $d-1$. We shall only consider   nontorsion framed sheaves, so we assume $\deg P_c(m)=d$. Let us fix a polynomial $\delta \in\mathbb{Q}[m]$ of degree $d-1$ with positive leading coefficient $\delta_{d-1}$.
For a framed sheaf $(E,\alpha:E\to\mathcal{F})$ of rank $\rk E>0$, denote $$\deg E:=c_1(E)\cdot H^{d-1}\,,\quad
\mu(E)=\deg E/\rk E\,,$$
$$
\deg (E,\alpha):=\deg E-\varepsilon(\alpha)\delta_{d-1}\,,\quad
\mu(E,\alpha):=\deg (E,\alpha)/\rk E,
$$
where $\varepsilon(\alpha):=1$ if $\alpha\ne0$
and  $\varepsilon(\alpha):=0$ otherwise.
Recall that a framed sheaf $(E,\alpha:E\to\mathcal{F})\in Y$ is called $\mu$-{\it (semi)stable
with respect to  $\delta_{d-1}$} in the sense of Huybrechts-Lehn \cite[Def. 1.8]{HL1} if $\ker \alpha$ is torsion free, and for all framed
subsheaves
$(E',\alpha')$ of $(E,\alpha)$, where $0\le\rk(E')\leq\rk E$ and
$\alpha':E'\hookrightarrow E\overset{\alpha}\to\mathcal{F}$ is the induced framing, one has
$$\rk E'\cdot\deg (E,\alpha)- \rk E\cdot\deg (E',\alpha')\underset{(\ge)}>0\,.$$ (If
$\rk E'>0$,   the latter inequality can be written as
$\mu(E',\alpha')\underset{(\le)}<\mu(E,\alpha)$.)

These (semi)stability notions for framed sheaves behave very much like the usual notions for coherent sheaves.
In particular,
\begin{enumerate}\item
the usual implications between the various notions of (semi)stability and $\mu$-(semi\-stability) hold also in the framed case, cf.~Section 1 of \cite{HL1}:
$$\mu\mbox{-stable}\quad \Rightarrow \quad \mbox{stable}
\quad \Rightarrow \quad\mbox{semistable}\quad \Rightarrow\quad\mu\mbox{-semistable}.$$
So, if we denote by $\mathcal{S}^{ss}(c,\delta)$ and $\mathcal{S}^{\mu ss}(c,\delta_{d-1})$ the families of all
framed  sheaves $(E,\alpha)$ of class $c$ on $X$, with $\alpha\ne 0$,  that are
semistable with respect to the polynomial $\delta$ and
$\mu$-semistable with respect to
$\delta_{d-1}$ (shortly: {\it $\mu$-semistable}), respectively,  one has the inclusion
\begin{equation}\label{ss in mu ss}
\mathcal{S}^{ss}(c,\delta)\subset\mathcal{S}^{\mu ss}(c,\delta_{d-1});
\end{equation}
\item there are restriction theorems of the Mehta-Ramanathan type \cite{Sala1};
\item   ($\mu$-)semistability is an open condition; in Proposition \ref{openness} below we prove
this for $\mu$-semistability.
\item The family of $\mu$-semistable framed sheaves with fixed numerical data is bounded, see Proposition \ref{bounded}.
\end{enumerate}

\begin{prop}\label{openness}
Let $(X,H)$, $D$, $\F$ be as above. Let $S$ be a noetherian scheme, and $\mathcal E$ a sheaf on $S\times X$ which is flat over $S$ and is of relative dimension $d=\dim X$ over $S$. Let $ \LL$ be an invertible sheaf on $S$, and $\alpha:
 \LL\otimes \mathcal E\to \mathcal O_S\boxtimes \F$ a framing of $\mathcal E$ as in Definition \ref{relfram}. Let us fix some rational number $\delta_{d-1} >0$.

Then the locus of points $s\in S$ for which $(\mathcal E_s,\alpha_s=\alpha|_{\{s\}\times X})$ is $\mu$-semistable with respect to $\delta_{d-1}$ is open.
\end{prop}

\begin{proof}
We argue along the lines of the proof of Proposition 2.3.1 in
\cite{HLbook}. Let $P=P(\E_s)$ be the Hilbert polynomial of the sheaves $\E_s=:\E|_{\{s\}\times X}$.
By $\hat\mu (P)$ we denote the hat-slope of a polynomial, $\hat\mu (P)=a_{d-1}/a_d$ if $P(n)=a_d\frac{n^d}{d!}+a_{d-1}\frac{n^{d-1}}{(d-1)!}
+\ldots +a_0$, $a_d\neq 0$. For a sheaf $E$, we denote $a_i(E)$ the coefficient $a_i$ in the above representation of the Hilbert polynomial $P=P(E)$, and $\hat\mu(E):= \hat\mu (P(E))$. The hat slope of a $d$-dimensional sheaf is related to the usual slope
by the formula $\hat\mu=\frac1{\deg X}\left(\mu-\frac12 K_X\cdot H^{d-1}\right)$.

Consider the exact triples
\begin{equation}\label{F'F''}
0\to F'\to\E_s\to F'' \to 0
\end{equation}
over all $s\in S$.
In the case when $F'\subset \ker\alpha_s$, we can also associate with \eqref{F'F''} the exact triple
\begin{equation}\label{F'F''a}
0\to F'\to\ker\alpha_s \to F_*'' \to 0\ .
\end{equation}
If the sheaf $F'$ in one of the triples \eqref{F'F''} or \eqref{F'F''a} destabilizes $\E_s$,
then we can replace it by its framed saturation, and it will still destabilize $\E_s$. Recall that
the framed saturation of $F'\subset \E_s$ is the saturation in $\ker\alpha_s$ if $F'\subset \ker\alpha_s$,
and the saturation in $\E_s$ if $F'\not\subset \ker\alpha_s$.

Consider first the case when $F'$ is a saturated destabilizing subsheaf such that $\dim F''<d$. Then we have $F'=\ker\alpha_s$,
$F''\simeq\im\alpha_s$, and $a_{d-1}(\im\alpha_s)<\delta_{d-1}$, or equivalently
$a_{d-1}(\F/\im\alpha_s)>a_{d-1}(\F)-\delta_{d-1}$. The set of the Hilbert polynomials $P_i$ of the
quotients $\F/\im\alpha_s$ as $s$ runs over $S$ is finite, and the Quot schemes $\Quot_{X}(\F, P_i)$ are projective, so the locus
\begin{multline*}
S_i=\{s\in S\ | \ \mbox{there exists $i$ such that $a_{d-1}(P_i)>a_{d-1}(\F)-\delta_{d-1}$ and}\
P(\F/\im\alpha_s)=P_i\}
\end{multline*}
is closed. Thus the locus
$$
S_*=\bigcup_i S_i
$$
of points of $S$ over which $(\E_s,\alpha_s)$ has a $\mu$-destabilizing subsheaf with torsion quotient $F''$ is closed.

Consider now the two sets of Hilbert polynomials of quotients $F''$ corresponding to the framed saturated destabilizing subsheaves for which $\dim F'' =d$:
\begin{multline*}
A_1 =\{P''\ | \  \deg P''=d,\  \hat\mu (P'')< \hat\mu (P)-
\tfrac{\delta_{d-1}}{r\deg X},\
\mbox{and there exist $s\in S$ and an exact}\\
\mbox{triple \eqref{F'F''} with $F''$ torsion free, such that $P(F'')=P''$}\},
\end{multline*}
\begin{multline*}
A_2 =\{P''\ | \  \deg P''=d,\  \hat\mu (P'')< \hat\mu (P) +\tfrac{\delta_{d-1}}{\deg X}(\tfrac{1}{r''}-\tfrac1r),\
\mbox{and there exist $s\in S$ and} \\
F'\subset \ker\alpha_s\  \mbox{such that $F_*''$ is torsion free of rank $r''$ and $P(F'')=P''$}\}.
\end{multline*}
As the families of sheaves $\{\E_s\}_{s\in S}$ and $\{\ker\alpha_s\}_{s\in S}$ are bounded, these sets are finite  by \cite{Gro}, Lemma 2.5.

The relative Quot schemes
$\pi: Q(P'')=\Quot_{S\times X/S}(\E, P'')\to S$ are projective, so their images $S(P''):=\pi(Q(P''))$
are closed. Some of the points of $S(P'')$ may correspond to non-saturated destabilizing subsheaves. But
the hat-slope of the Hilbert polynomial of $F''$ may only decrease when we replace $F'$ by its framed saturation.
Hence a point in $S(P'')$ represented by a quotient $F''$ (or $F''_*$) with torsion for some polynomial $P''\in A_i$ is also represented by a {\em torsion-free} quotient $\tilde F''$ (or $\tilde F''_*$) from $Q(\tilde P'')$
for another polynomial $\tilde P''\in A_i$ with the same $i=1, 2$.
Hence the locus of points $s\in S$ for which
$\E_s$ is not $\mu$-semistable is the union
\begin{equation}\label{union}
Z\ \ = \ \ \left(\bigcup_{P''\in A_1} S^\circ(P'')\right)\ \bigcup\ \left(\bigcup_{P''\in A_2} S(P'')\right)\bigcup S_*.
\end{equation}
Here $ S^\circ(P'')= \pi(Q^\circ(P''))$, where $Q^\circ(P'')$ is the open subset of  $Q(P'')$ consisting
of the triples \eqref{F'F''} with $F'\not\subset\ker\alpha_s$. Remark that if $[0\to F'\to\E_s\to F'' \to 0]\in Q(P'')\setminus Q^\circ(P'')$ for some $P''\in A_1$, then $F'\subset \ker \alpha_s$ and
$$\hat\mu (P'')< \hat\mu (P)-
\tfrac{2\delta_1}{r\deg X} < \hat\mu (P) +\tfrac{2\delta_1}{\deg X}(\tfrac{1}{r''}-\tfrac1r)\,,$$
so that
$P''\in A_2$. Thus if we replace $ S^\circ(P'')$ by $ S(P'')$  in \eqref{union}, the union on the r.~h.~s. will not change, and we have
$$
Z=\bigcup_{P''\in A_1\cup A_2} S(P'')\ \ \bigcup \ S_*
$$
Thus $Z$ is a union of finitely many closed subsets of $S$, and is closed in $S$.
\end{proof}

\begin{prop}\label{bounded}
The family $\mathcal{S}^{\mu ss}(c,\delta)$ is bounded.
\end{prop}

\begin{proof}
The sheaves $E$ from the pairs $(E,\alpha)\in \mathcal{S}^{\mu ss}(c,\delta)$
may have torsion.
We use the following trick of Huybrechts--Lehn (Remark 1.9 and  Lemma 2.5 from \cite{HL1})
 to replace them with torsion-free ones.
Let $\hat\F$ be any locally-free sheaf with a surjection $\phi:\hat\F\to\F$
and $\hat E=E\times_\F\hat\F$. Then $\hat E$ is torsion free,
and there is an exact triple $0\to \K\to\hat E\stackrel{\phi_E}{\relbar\joinrel\longrightarrow}  E\to 0$, where $\phi_E$ is the morphism induced by $\phi$, and $\K=\ker\phi_E$.
Thus if we fix $\hat \F$ and $\phi$, then $P_{\hat E}=P_c+P_{\K}$
does not depend on $(E,\alpha)$.

Let now $\hat F$ be any nonzero subsheaf of $\hat E$. Then $\rk \hat F>0$, as
$\hat E$ is torsion free. We have an exact triple
$0\to \K_F\to\hat F\to F\to 0$, where $F=
\phi_E(\hat F)$ and $\K_F=\ker (\phi_E|_{\hat F})$.
By the $\mu$-semistability of $(E,\alpha)$, we have
$\deg (F)\le\rk F\cdot (\mu (E)+\delta_{d-1})$. Hence
$$
\mu (\hat F)=\frac{\deg F+\deg\K_F}{\rk \hat F}\le
\frac{\rk F\cdot(\mu_c+\delta_{d-1})+\rk\K_F\cdot\mu_{\max}(\K)}{\rk \hat F}\ ,
$$
where $\mu_{\max}$ stands for the slope of the maximal
destabilizing subsheaf.

This shows that $\mu_{\max}(\hat E)$ is uniformly
bounded as $(E,\alpha)$ runs through $\mathcal{S}^{\mu ss}(c,\delta)$.
Hence by a theorem of
Le Potier-Simpson \cite[Thm. 3.3.1]{HLbook}, there exist constants $C_0,\ldots,C_d$ and an
$\hat E$-regular sequence of hyperplane sections $H_1,\ldots,H_d\in|\mathcal{O}_X(H)|$ such that
$\ h^0(\hat E|_{X_\nu})\le C_\nu$, where
$X_\nu=H_1\cap\ldots\cap H_{d-\nu}$,   $\nu=0,\ldots, d$. See \cite{HLbook}, Def.~1.1.11
for the definition of a regular sequence of sections of a line bundle with respect to a given sheaf.
Now apply Kleiman's boundedness  criterion
\cite[Thm. 1.7.8]{HLbook} to obtain the boundedness
of the family of the sheaves $\hat E$ associated with the pairs $(E,\alpha)$
from $\mathcal{S}^{\mu ss}(c,\delta)$. The boundedness of the family of
the pairs $(E,\alpha)$ themselves then follows by the same argument as in the proof
of Lemma 2.5 in \cite{HL1}.
\end{proof}

By Proposition \ref{bounded} and semicontinuity we can fix a sufficiently large number $m$ such
that for each pair $(E,\alpha)$ in $\mathcal{S}^{\mu ss}(c,\delta)$ the sheaf $E$ is
$m$-regular. We define now $\widetilde{R}^{\mu ss}(c,\delta)$ as the locally closed subscheme of the scheme
$$
Y=\fquot ,
$$
with $\HH=V\otimes\cO_X(-m)$ and $\dim V=P_c(m)$,
formed by the pairs $([g:\mathcal{H}\to E],[a:\mathcal{H}\to\mathcal{F}])$
such that
$(E,\alpha)\in\mathcal{S}^{\mu ss}(c,\delta)$,
is $\mu$-semistable with respect to  $\delta_{d-1}$,
where the framing $\alpha$ is defined by the relation $a=\alpha\circ g$,
and $g$ induces an isomorphism $V\to H^0(E(m))$.
By (\ref{ss in mu ss}) $\widetilde{R}^{\mu ss}(c,\delta)$ contains a subset
$R^{ss}(c,\delta)$ consisting of semistable pairs $(E,\alpha)$,
and it is known that $R^{ss}(c,\delta)$ is open in $\widetilde{R}^{\mu ss}(c,\delta)$. We
denote by
$
R^{\mu ss}(c,\delta)
$
the closure of $R^{ss}(c,\delta)$ in $\widetilde{R}^{\mu ss}(c,\delta)$.

\subsection{Choosing a semiample sheaf $\LL(n_1,n_2)$ on
$R^{\mu ss}(c,\delta)$}

From now on we assume that:\\
(i) $X$ is a surface (i.e. $d=2$),\\
(ii) $\mathcal{F}$ is an $\mathcal{O}_D$-module, where $D\subset X$ is a fixed big and nef curve,\\
(iii) $\deg P_c(m)=2$.

We will identify $K(X)_{\rm num}$ with the group $\mathbb{Z}\oplus NS(X)\oplus\mathbb{Z}$ via the map sending the class
$[E]$ of a sheaf $E$ to the triple $(\rk(E), c_1(E),c_2(E))$.
We fix a polynomial
$$
\delta(m)=\delta_1m+\delta_0\in\mathbb{Q}[m]\ {\rm with}\ \delta_1>0.
$$
For any framed sheaf $(E,\alpha)$ on $X$ we set
$$
P_{(E,\alpha)}(l):=P_E(l)-\varepsilon(\alpha)\delta(l)
$$
If $(E,\alpha)\in\mathcal{S}^{\mu ss}(c,\delta)$, then $\alpha\ne 0$, so that $\epsilon(\alpha)=1$, and  there is a surjective   morphism
$V\otimes\mathcal{O}_X(-m)\to E$. Since the family of subsheaves $E'$ of $E$
generated by all subspaces  $V'$ of $V$ is bounded, the set $\mathcal{N}_{(E,\alpha)}$ of their Hilbert
polynomials $P_{E'}$ is finite. Hence, since the family $\mathcal{S}^{\mu ss}(c,\delta)$ is
bounded, the set
$$\mathcal{N}_X(c,\delta):=
\underset{(E,\alpha)\in\mathcal{S}^{\mu ss}(c,\delta)}\bigcup\mathcal{N}_{(E,\alpha)}$$
is finite.

Now for each polynomial $B\in\mathcal{N}_X(c,\delta)$, where $B=P_{E'}$, for $E'$
a subsheaf of some framed sheaf
$(E,\alpha)\in\mathcal{S}^{\mu ss}(c,\delta)$, defined by a subspace $V'$ of $V$,
together with the induced framing $\alpha'$, we denote
\begin{equation*}
G_B(l):=\dim V\left(1+\varepsilon(\alpha')
\frac{\delta(m)}{P_{(E',\alpha')}(m)}\right)P_{(E',\alpha')}(l)
-\dim V'\left(1+\frac{\delta(m)}{P_{(E,\alpha)}(m)}\right)P_{(E,\alpha)}(l).
\end{equation*}
Since the set $\{G_B|B\in\mathcal{N}_X(c,\delta)\}$ is finite, there exists a
rational number $\ell_0$ such that for any $\ell'\ge\ell_0$ the implication
\begin{equation}
G_B(\ell')>0\ \Rightarrow\ G_B(l)\ {\rm is\ positive\ for}\ l\gg0
\end{equation}
is true for all $B\in\mathcal{N}_X(c,\delta)$.

Fix an integer $k>0$ large enough to ensure that
\begin{equation*}
H^1(X,E(m-k))=0\ \ \mbox{for all} \ \ (E,\alpha)\in\mathcal{S}^{\mu ss}(c,\delta).
\end{equation*}
For any
$(E,\alpha)\in\mathcal{S}^{\mu ss}(c,\delta)$ there is a dense open subset
$|kH|^*$ in the linear system $|kH|$ consisting of the smooth curves $C$ which are transversal to the framing divisor $D$ and do not meet the singular locus of $(E,\alpha)$, that is, the locus of points $x\in S$ where $E$ is not locally free or $x\in D$, $\alpha|_x=0$.
For any curve $C\in|kH|^*$, we have $$
P_{c|_C}(l):=P_{E|C}(l)=P_{E}(l)-P_{E}(l-k)=P_c(l)-P_c(l-k)\ , \ \
P_{(E|_C,\alpha|_C)}= P_{c|_C} - \delta.
$$
Consider the rational functions
\begin{equation}\label{AX1}
A_X(l):={P_{(E,\alpha)}(l)}\frac{\delta(m)}{P_{(E,\alpha)}(m)}-\delta(l)\in\mathbb{Q}(l),
\end{equation}
\begin{equation}\label{AC}
A_C(l):={P_{(E|_C,\alpha|_C)}(l)}\frac{\delta_C}{P_{(E|_C,\alpha|_C)}(m)}-
\delta_C\in\mathbb{Q}(l),
\end{equation}
where, as before, $\delta(l):=\delta_1l+\delta_0$ and   we set $\delta_C:=k\delta_1$.
Let
$$
P_c(l)=p_2l^2+p_1l+p_0,\ \ \ p_i\in\mathbb{Q}.
$$
The equality
\begin{equation}\label{AX(l)=AC(tilde l)}
A_X(l)=A_C(\tilde{l}),
\end{equation}
considered as an equation in $\tilde{l}$, in view of (\ref{AX1}) and (\ref{AC}), yields
\begin{equation}\label{tilde(l)eq}
\tilde{l}=L(l):=A_X(l)\frac{p_2(2m-k)+p_1-\delta_1}{2p_2\delta_1k}+m
\end{equation}
For any $C\in|kH|^*$, set
$$\mathcal{H}_C:=V_C\otimes\mathcal{O}_C(-m),\qquad\dim V_C:=P_c(m)-P_c(m-k).$$
We have $$P_{c|_C}(l)=P_c(l)-P_c(l-k)=k(p_2(2l-k)+p_1).$$
Consider the Quot scheme $Y_C:= \mbox{$\mathbb Q$uot}(\mathcal{H}_C,P_{c|_C},\mathcal{F}|_C)$. For any
$(E,\alpha)\in Y$ and any $C\subset|kH|^*$,
consider the framed sheaf $(E|_C,\alpha|_C)$. The family of subsheaves $E'_C$
of $E|_C$ generated by all the subspaces $V_C'$ of $V_C$ is bounded, so that the set
$\mathcal{N}_{(E|_C,\alpha|_C)}$ of polynomials $P_{E'_C}$ is finite. Hence, since the family
$\mathcal{S}^{\mu ss}(c,\delta)$ is
bounded, the set
$$\mathcal{N}_C(c|_C,\delta_C):=
\underset{(E,\alpha)\in\mathcal{S}^{\mu ss}(c,\delta)}\bigcup\mathcal{N}_{(E|_C,\alpha|_C)}$$
is finite. On the other hand, the set
$$
\mathcal{N}(c|_C,\delta_C):=\underset{C\in|kH|^*}\bigcup\mathcal{N}_C(c|_C,\delta_C)
$$
is   finite as well.

Now for each polynomial $B\in\mathcal{N}(c|_C,\delta_C)$, where $B=P_{E'_C},\ E'_C$
a subsheaf of a sheaf $(E|_C,\alpha|_C)$ for some framed sheaf
$(E,\alpha)\in\mathcal{S}^{\mu ss}(c,\delta)$, defined by a subspace $V_C'$ of $V_C$,
together with the induced framing $\alpha'_C$, we denote
\begin{multline*}
\widetilde{G}_B(l):=\dim V_C\cdot\left(P_{(E'|_C,\alpha'|_C)}(l)+\varepsilon(\alpha'|_C)
\frac{\delta(m)}{P_{(E|_C,\alpha|_C)}(m)}\right) \\
-\dim V_C'\cdot\left(1+\frac{\delta_C(m)}{P_{(E|_C,\alpha|_C)}(m)}\right)P_{(E|_C,\alpha|_C)}(l).
\end{multline*}
Since the set $\{\widetilde{G}_B|B\in\mathcal{N}(c|_C,\delta_C)\}$ is finite, there exists a  rational number $\ell_{0C}$ such that for any $\ell'\ge\ell_{0C}$ the implication
\begin{equation}\label{biggerthanzero}
\widetilde{G}_B(\ell')>0\ \Rightarrow\ \widetilde{G}_B(l)\ {\rm is\ positive\ for}\ l\gg0
\end{equation}
is true for all $B\in\mathcal{N}(c|_C,\delta_C)$.

Now choose a number $\ell_X\ge\ell_0$ such that $L(\ell_X)\ge\ell_{0C}$,  where $\ell_{0C}$ was defined before formula \eqref{biggerthanzero} and $L(l)$ was defined earlier in
(\ref{tilde(l)eq}).
 Set
$
\ell_C:=L(\ell_X)
$.
By
(\ref{AX(l)=AC(tilde l)}) we have
\begin{equation}\label{AX2}
A_X(\ell_X)=A_C(\ell_C),\ \ \ \ \ \ \ \ell_X\ge\ell_0,\ \ \ \ell_C\ge\ell_{0C}.
\end{equation}
Let
$$\LL(n_1,n_2)=\left[\operatorname{pr}_1^\ast \lambda_{\tilde \Q}(u_1)^{\otimes n_1}
\otimes\operatorname{pr}_2^\ast\cO_{\PP}(n_2) \right]|_{R^{\mu ss}(c,\delta)}$$
where we set
\begin{equation}\label{frac}
\frac{n_1}{n_2}:=A_X(\ell_X)=\delta(m)\frac{P_{(E,\alpha)}(\ell_X)}{P_{(E,\alpha)}(m)}-\delta(\ell_X),
\end{equation}
and $\tilde\Q$ is the  universal quotient sheaf on $\quot$, see the last paragraph of Section \ref{Quot}.
This choice of the ratio $\frac{n_1}{n_2}$ will enable us to obtain the isomorphism \eqref{cong}.

Now one has the following analogues of theorems of Mehta and Ramanathan \cite{MR1,MR2}.
\begin{thm}\label{Meh-Ram}
Let $(E,\alpha:E\to\mathcal{F})\in\mathcal{S}^{\mu ss}(c,\delta)$ be a $\mu$-semistable framed sheaf of positive rank.
Then for all sufficiently big $k$, and for a generic curve
$C\in|kH|$, the framed sheaf $(E|_C,\alpha|_C)$ is $\mu$-semistable on $C$
with respect to  $\delta_C$.
\end{thm}
\begin{proof}
See \cite [Theorem 67]{Sala1}.
\end{proof}
\begin{thm}\label{Meh-Ram2}
In conditions of Theorem \ref{Meh-Ram} let the framing sheaf $\mathcal{F}$ be a locally
free $\mathcal{O}_D$-module
and let $(E,\alpha)$ be a $\mu$-stable framed sheaf such that $E$ is locally free in a neighborhood of $D$ and $\alpha|_D:\ E|_D\to\mathcal{F}$ is an isomorphism.
Then for all sufficiently big $k$, and for a generic curve
$C\in|kH|$, the framed sheaf $(E|_C,\alpha|_C)$ is $\mu$-stable on $C$
with respect to  $\delta_C$.
\end{thm}
\begin{proof}
See \cite [Theorem 74]{Sala1}.
\end{proof}

\begin{prop}\label{gener}
For $\nu\gg 0$ the line bundle $\LL(n_1,n_2)^\nu$ on $R^{\mu ss}$ is generated
by its $SL(V)$-invariant sections.
\end{prop}

\begin{proof}
Let $S$ be a noetherian scheme parametrizing a flat family
$(\mathbf{E},\alpha_\mathbf{E})$
of $\mu$-semistable framed sheaves $(E,\alpha:E\to\FFF)$ on $X$ with numerical K-theory class $c=(r,\xi,c_2)$.
Let $C\in|kH|$ be a general curve and $k\gg0$. Then $C$ is smooth and
transversal to $D$, and
the restriction of $(\mathbf{E},\alpha_\mathbf{E})$ to $S\times C$ yields a family
$(\mathcal{E},\alpha_{\mathcal{E}})$
of framed sheaves $(E_C,\alpha_C:E_C\to \mathcal{F}|_C)$ on the curves $C$. By Proposition \ref{openness} and Theorem \ref{Meh-Ram}, the general element in this family is $\mu$-semistable.
Let $M_C:=M^{ss}(c|_C,\delta_C)$ be the moduli space of framed sheaves on $C$ with numerical class   $c|_C=i^*c$ that are
 semistable with respect to  $\delta_C$. Note that, since $C$ is a curve, semistability
coincides with $\mu$-semistability. By Theorem \ref{Meh-Ram}   a rational map $S\dasharrow M_C$ is well defined.

The class $c|_C$ is
uniquely determined by its rank and by $\xi|_C$. Let $m'$ be a large positive
integer, $P':=P_{c|_C}$, let $V_C$ be a vector space of dimension $P'(m')$, let $\mathcal{H}':=V_C\otimes\mathcal{O}_C(-m')$ and let
$Q_C\subset{\rm Quot}_C(\mathcal{H}',P')$ be the closed subset of quotients with first Chern class
$\xi|_C$, together with the universal quotient
$\mathcal{O}_{Q_C}\boxtimes\mathcal{H}'\to\mathcal{E}'$.
Furthermore, let
$\mathbb{P}_C=\mathbb{P}\big(\Hom(V_C,H^0(C,\mathcal{F}(m')|_C))^\ast\big)$,
so that a point $[a]\in\mathbb{P}_C$ corresponds to a morphism $a:\mathcal{H}'\to \mathcal{F}|_C$.
Consider the closed subscheme
$Y_C=\mbox{$\mathbb Q$uot}(\mathcal{H}',P',\mathcal{F}|_C)$ of $Q_C\times\mathbb{P}_C$
with projections
$Q_C\stackrel{p_1}{\leftarrow \joinrel\relbar}Y_C\stackrel{p_2}{\relbar\joinrel\rightarrow}\mathbb{P}_C$,
defined similarly to the scheme $Y$ above. Clearly, $\operatorname{SL}(V_C)$ acts on
$Y_C$.
Denote
$\deg C=C\cdot H$, and consider the line bundle
$$
\mathcal{L}'_0(n_1,n_2k):=p_1^*\lambda_{\mathcal{E}'}(u_0(c|_C))^{n_1\deg C}\otimes
p_2^*\mathcal{O}_{\mathbb{P}_C}(n_2k)
$$
on $Y_C$. If $m'$ is sufficiently large, the following results hold (see \cite{HL1} Proposition 3.2).

\begin{lemma}\label{newlemma1}  Given a point $([g:\mathcal{H}'\to E_C],[a:\mathcal{H}'\to \mathcal{F}|_C])\in Y_C$,
the following assertions are equivalent:
\begin{enumerate}
\item  $(E_C,[a])$ is a semistable pair and $V_C\to H^0(E_C(m'))$ is an isomorphism.
\item $([g],[a])$ is a semistable point in $Y_C$ for the action of $\operatorname{SL}(V_C)$ with respect to
the canonical linearization of $\mathcal{L}'_0(n_1,n_2k)$.
\item There is an integer $\nu$ and an $\operatorname{SL}(V_C)$-invariant section $\sigma$ of
$\mathcal{L}'_0(n_1,n_2k)^\nu$ such that $\sigma([g],[a])\ne0$.
\end{enumerate}
\end{lemma}

Jordan-H\"older filtrations for semistable framed sheaves were introduced in \cite{HL1}, Proposition 1.13, and the ensuing notion of S-equivalence was given there in Definition 1.14.  In Section \ref{descr_mor}
we shall also use the notion of a $\mu$-Jordan-H\"older filtration of a framed sheaf
$(E,\alp)$.  It is constructed in a similar way, see \cite{Sala1}, Definition 65 and Theorem 66.
We call $(E,\alpha)$ \ \ $\mu$-polystable if $E$ has a filtration $0=E_0
\subset E_1\subset\ldots\subset E_n=E$ such that: (i)   $E$  is isomorphic to the graded object
$\bigoplus\limits_{i=1}^nE_i/E_{i-1}$;
(ii) the filtration $\ldots\subset(E_i,\alpha|_{E_i})
\subset(E_{i+1},\alpha|_{E_{i+1}})\subset\ldots$
is a $\mu$-Jordan-H\"older filtration of
$(\EEE,\alp)$. For a given framed sheaf $F=(E,\alpha_E)$  we will denote by $gr^\mu F$ the associated graded
$\mu$-semistable framed sheaf $\bigoplus\limits_{i=1}^n(E_i/E_{i-1},\alpha_i)$ where
$\bigoplus\limits_{i=1}^nE_i/E_{i-1}$ is the graded sheaf associated
with a  $\mu$-Jordan-H\"older filtration $0=E_0\subset E_1\subset\ldots\subset E_n=E$ of $E$ and where
$\alpha_i:E_i/E_{i-1}\to\mathcal{F}$ are the induced framings
(see \cite[Section 1]{HL1}). If, moreover, $E$ is locally free along $D$,
then by $(gr^\mu F)^{\dual\dual}$ we will understand the graded $\mu$-semistable framed sheaf
$\underset{i}\oplus((E_i/E_{i-1})^{\dual\dual},\alpha_i)$.
\begin{rema}\label{S-eq}
Note that in the case when $\dim X=1$, the $\mu$-Jordan-H\"older filtration of a framed semistable sheaf $F$ on $X$ coincides
with the Jordan-H\"older filtration of $F$; hence, the two framed sheaves $F_1$ and $F_2$ on $X$ are S-equivalent if
and only if their associated graded objects $gr^\mu F_1$ and $gr^\mu F_2$ are isomorphic.
\end{rema}

The following result is essentially contained in Proposition 3.3 of \cite{HL1}.
\begin{lemma} \label{newlemma2}
Two points $([g_j:\mathcal{H}'\to E_{jC}],[a_j:\mathcal{H}'\to \mathcal{F}|_C]),\ j=1,2$ are
separated by an $\operatorname{SL}(V_C)$-invariant section in some tensor power of
$\mathcal{L}'_0(n_1,n_2k)$ if and only if either
both are semistable points but the corresponding framed sheaves $(E_{1C},\alpha_{1C})$ and
$(E_{2C},\alpha_{2C})$ are not $S$-equivalent,  or one of them is
semistable but the other is not.
\end{lemma}

Consider now the exact sequence
\begin{equation}\label{restrict}
0\to\mathbf{E}\otimes\big(\mathcal{O}_S\boxtimes\mathcal{O}_X(-k)\big)\to\mathbf{E}\to\mathcal{E}\to0.
\end{equation}
Assume that $m'$ is big enough so that not only
the results in Lemmas  \ref {newlemma1} and \ref{newlemma2}   hold, but
one also has

{\em  $\mathcal{E}_s$ is $m'$-regular for all $s\in S$.}

\noindent Then $p_*(\mathcal{E}(m'))$ is a locally-free $\mathcal{O}_S$-module of rank $P'(m')$,  where
$\mathcal{E}(m')=\mathcal{E}\otimes\mathcal{O}_S\boxtimes\mathcal{O}_C(m')$
and $p:S\times C\to S$ is the projection. Let $\widetilde{S}:=\mathbb{P}({\rm Isom}(V_C\otimes\mathcal{O}_S,p_*(\mathcal{E}(m')))^*)$, \ \
$\pi:\widetilde{S} \to S$ be  the associated projective frame bundle and $\pi_C:\widetilde{S}\times C\to S\times C$ and $\tilde p=\widetilde{S}\times C\to \widetilde{S}$
the natural maps. On $\tilde S\times C$ there is a universal quotient
$\mathbf{g}:\mathcal{O}_{\widetilde{S}}\boxtimes\mathcal{H}'\twoheadrightarrow
\pi_C^*\mathcal{E}$
and a
framing
$\Psi_\mathcal{E}:\pi_C^*\mathcal{E}\otimes\tilde p^*(\LLL_{\widetilde{S}})
\to
\pi_C^*(\mathcal{O}_S\boxtimes \mathcal{F}|_C)$ for some invertible sheaf $\LLL_{\widetilde{S}}$
on $\widetilde{S}$,
and these data induce, by Proposition \ref{univprop}, a $\operatorname{SL}(P'(m'))$-invariant morphism
$$
\mathbf{f}_\mathcal{E}:\widetilde{S}\to Y_C.
$$
By analogy with  \cite[Prop. 8.2.3]{HLbook} and using the relations
(\ref{AX2}) and (\ref{frac}), we obtain the isomorphism of line bundles
\begin{equation}\label{cong}
\mathbf{f}_\mathcal{E}^*\mathcal{L}'_0(n_1,n_2k)\cong\pi^*\mathcal{L}(n_1,n_2)^{\otimes k}.
\end{equation}

Now set $S=R^{\mu ss}(c,\delta)$. The group $\operatorname{SL}(V)$ acts on $S$, hence also on
$\widetilde{S}$. Thus we have an action of $\operatorname{SL}(V)\times\operatorname{SL}(V_C)$
on $\widetilde{S}$ and by construction the morphism $\mathbf{f}_\mathcal{E}$ is
$\operatorname{SL}(V)\times\operatorname{SL}(V_C)$-invariant, where $SL(V)$ acts trivially on $Y_C$.
Take an arbitrary $\operatorname{SL}(V_C)$-invariant section $\sigma$ of
$\mathcal{L}'_0(n_1,n_2k)^{\otimes\nu}$. Then $\mathbf{f}_\mathcal{E}^*\sigma$ is a
$\operatorname{SL}(V)\times\operatorname{SL}(V_C)$-invariant section. Therefore, since $\pi$
is a principal $\operatorname{PSL}(V_C)$-bundle, this section descends to a
$\operatorname{SL}(V)$-invariant section of the line bundle
$\mathcal{L}(n_1,n_2)^{\otimes\nu k}$. We thus obtain a monomorphism
\begin{equation}\label{mono}
s_\mathcal{E}:H^0(Y_C,\mathcal{L}'_0(n_1,n_2k)^{\otimes\nu})^{\operatorname{SL}(V_C)}\to
H^0(S,\mathcal{L}(n_1,n_2)^{\otimes\nu k})^{\operatorname{SL}(V)}.
\end{equation}
By analogy  with \cite[Lemma 8.2.4]{HLbook}, and  using \cite[Prop. 3.1-3.3]{HL1}, we obtain
the following Lemma.
\begin{lemma}\label{separate}
1. If $s\in R^{\mu ss}(c,\delta)$ is a point such that
$(E_s|_C,\alpha_s|_C:E_s|_C\to \mathcal{F}|_C)$
is semistable with respect to $\delta_C$, there is a
$\operatorname{SL}(V)$-invariant section
$\bar{\sigma}\in H^0(R^{\mu ss}(c,\delta),
\mathcal{L}(n_1,n_2)^{\otimes\nu k})^{\operatorname{SL}(V)}$
such that $\bar{\sigma}(s)\ne0$.

2. If $s_1$ and $s_2$ are the two points in $R^{\mu ss}(c,\delta)$ such that $E_{s_1}|_C$ and
$E_{s_2}|_C$
are both semistable but not $S$-equivalent, or one of them
is semistable and the other is not,  then  for some $\nu$ there are
$\operatorname{SL}(V)$-invariant sections of
$\mathcal{L}(n_1,n_2)^{\otimes\nu k}$  that separate $s_1$ and $s_2$.
\end{lemma}

Proposition \ref{gener} now follows from the first assertion of Lemma \ref{separate}.
\end{proof}

\bigskip
\section{The Uhlenbeck-Donaldson compactification   for framed
sheaves}\label{Uhlenbeck comp}

\subsection{Construction of $M^{\mu ss}(c,\delta)$}
By Proposition \ref{gener}, the sheaf $\mathcal{L}(n_1,n_2)^\nu$ is generated by  its invariant
sections. Thus we can find a finite-dimensional subspace
$W\subset W_\nu:=H^0(R^{\mu ss},\mathcal{L}(n_1,n_2)^\nu)^{\operatorname{SL}(V)}$
that generates $\mathcal{L}(n_1,n_2)^\nu$. Let $\phi_W:R^{\mu ss}(c,\delta)\to\mathbb{P}(W)$ be
the induced $\operatorname{SL}(P_c(m))$-invariant morphism.

\begin{prop}\label{MW proj}
$M_W:=\phi_W(R^{\mu ss}(c,\delta))$ is a projective scheme.
\end{prop}
The proof of this Proposition goes as in \cite[Prop. 8.2.5]{HLbook}, by
using the following Proposition, which generalizes  a classical result by Langton.
\begin{prop}\label{Langton}
Let $(R,\mathfrak m=(\pi))$ be a discrete valuation ring with residue field $k$ and quotient field $K$
and let $X$ be a smooth projective surface over $k$. Let $\mathcal{E}=(\mathcal{E},\alpha)$
be an $R$-flat family of framed sheaves on $X$ such that $\mathcal{E}_K=K\otimes_R\mathcal{E}$ is a $\mu$-semistable
framed sheaf. Then there is a framed sheaf $(E,\alpha^E)$ such that
$(E_K,\alpha^E_K)=(\mathcal{E}_K,\alpha_K)$ and
$(E_k,\alpha^E_k)$ is $\mu$-semistable.
\end{prop}

Before proceeding to the proof of Proposition \ref{Langton} we prove the following auxiliary Lemma.

\begin{lemma}\label{Langton0}
Let $(R,\mathfrak m=(\pi))$ be a discrete valuation ring with residue field $k$ and quotient field $K$, let
$T=\Spec(R)$ and let $X$ be a smooth projective variety over $k$.
Let $\mathcal{E}=(\mathcal{E},\alpha:\mathcal{E}\to\mathcal{F})$
be a $T$-flat family of framed sheaves on $X$ such that $(\mathcal{E}_K=K\otimes_R\mathcal{E},\alpha_K)$ is a
$\mu$-semistable framed sheaf. Then there is a framed sheaf $(\widetilde{\mathcal{E}},\widetilde{\alpha})$ such that
$(\widetilde{\mathcal{\mathcal{E}}}_K,\widetilde{\alpha}_K)=(\mathcal{E}_K,\alpha_K)$ and $\ker({\widetilde{\alpha}}_k)$ has no torsion: $T(\ker({\widetilde{\alpha}}_k))=0$.
\end{lemma}
\begin{proof}
If $T(\ker(\alpha_k))=0$, then set $(\widetilde{\mathcal{E}},\widetilde{\alpha})=(\mathcal{E},\alpha)$, and we are done.
Thus assume that $T(\ker\alpha_k)\ne0$.
Choose an epimorphism
$\epsilon:\widehat{\mathcal{F}}\to\mathcal{F}$, where $\widehat{\mathcal{F}}$ is a locally-free
$\mathcal{O}_{T\times X}$-sheaf, so that $\mathcal{B}:=\ker(\epsilon)$ is torsion free. By the
same trick of Huybrechts and Lehn as was used in the proof of Proposition \ref{bounded}, we obtain
from $(\mathcal{E},\alpha)$ a framed sheaf $(\widehat{E},\widehat{\alpha})$ on $T\times X$
such that $T(\widehat{E})=T(\ker\alpha)$. As $\widehat{\mathcal{F}}$ is locally free,
tensoring with $k$ or $K$ commutes with the construction of $\widehat{E}$,
so that $T(\widehat{E}_k)=T(\ker(\alpha_k))$ and $T(\widehat{E}_K)=T(\ker(\alpha_K))$.
By the $\mu$-semistability of
$(\mathcal{E}_K,\alpha_K)$, we have $T(\ker(\alpha_K))=0$, so $\widehat{E}_K$ is torsion free.
One also easily verifies that $\ker(\alpha_k)=\ker(\widehat{\alpha}_k)$.

Let $Y$ be the support of
$T(\ker(\alpha_k)) = T(\ker(\widehat{\alpha}_k))$ in $T\times X$, and let $E':= j_*(\widehat{E}|_U)$,
where $U = T\times X - Y$, and
$j : U\to T\times X$ is the natural inclusion.
Then $E'$ has no $T$-torsion and is $T$-flat. In particular,
the fibre $E'_k$ has the same Hilbert polynomial as $E_k$.
The canonical morphism $\widehat{E}\to E'$ induces a
morphism $\phi: \widehat{E}_k\to E'_k$ which is an isomorphism outside $Y$. Since $\widehat{\mathcal{F}}$ is
locally free, hence normal, $\widehat{\alpha}$ defines a framing $\alpha':E'\to\widehat{\mathcal{F}}$ which
coincides with $\widehat{\alpha}$ on $U$. Now let $\widetilde{\mathcal{E}}$ be the cokernel of the composition
$\mathcal{B}\to\widehat{E}\xrightarrow{can}E'$, together with the induced framing
$\widetilde{\alpha}:\widetilde{\mathcal{E}}\to\mathcal{F}$. Then $(E'_k, \alpha'_k)$ is exactly the framed sheaf constructed from $(\widetilde{\mathcal{E}}_k,\widetilde{\alpha}_k)$ by the Huybrechts--Lehn trick
via the surjection $\widehat{\mathcal{F}}_k\to{\mathcal{F}}_k$, so $T(E'_k)=T(\ker{\alpha'}_k)$, and as above, we deduce that $\ker(\widetilde{\alpha}_k)=\ker({\alpha'}_k)$.

By construction, $(\widetilde{\mathcal{\mathcal{E}}}_K,\widetilde{\alpha}_K)=(\mathcal{E}_K,\alpha_K)$.
The same argument as in the proof of Lemma
1.7 in \cite{Sim} shows that $E'_k$ is torsion free, hence
$T(\ker(\widetilde{\alpha}_k))=T(\ker({\alpha'}_k))=0$, and we are done.
\end{proof}

\textit{Proof of Proposition \ref{Langton}.}
If $(\mathcal{E}_k,\alpha_k)$ is $\mu$-semistable, then we are done. Assume that this is not the case.
By Lemma \ref{Langton0}, we may also assume that $T(\ker{{\alpha}}_k)=0$.
Setting $(\mathcal{E}^0,\alpha^0)=({\mathcal{E}},{\alpha})$,
we will define a descending sequence of framed sheaves
$({\mathcal{E}},{\alpha})=(\mathcal{E}^0,\alpha^0)\supset
(\mathcal{E}^1,\alpha^1)\supset(\mathcal{E}^2,\alpha^2)\supset...$, such that
$\mathcal{E}^n_K=\mathcal{E}_K $ and $(\mathcal{E}^n_k,\alpha^n_k)$ is not $\mu$-semistable for all $n$.

Let $(B^0,\alpha_{B^0})$ be the maximal $\mu$-destabilizer of
$(\mathcal{E}^0_k,\alpha^0_k)$, where $\alpha_{B^0}$ is the induced framing.
As $T(\ker\alpha^0_k)=0$, it follows from \cite[Theorem 6.6]{Sala1} that $B^0$ is $\mu$-semistable and  framed-saturated.

Suppose that, for $n\ge0$, the framed sheaf $(\mathcal{E}^n,\alpha^n)$ and its
saturated maximal $\mu$-destabilizer $(B^n,\alpha_{B^n})$ have been defined.
Let $G^n = \mathcal{E}^n_k/B^n$ together with the induced framing $\alpha_{G^n}:G^n\to\mathcal{F}_k$,
and let $\mathcal{E}^{n+1}$ be the kernel of the composition
$\mathcal{E}^n\to\mathcal{E}^n_k\to G^n$
with the induced framing $\alpha^{n+1}:\mathcal{E}^{n+1}\to\mathcal{F}$.
As a subsheaf of an $R$-flat sheaf, $\mathcal{E}^{n+1}$ is $R$-flat. We thus have two exact sequences
\begin{equation}\label{2 triples}
0\to B^n\to \mathcal{E}^n_k\to G^n\to 0\ \ \ \text{and}\ \ \ 0\to G^n\to\mathcal{E}^{n+1}_k\to B^n\to 0.
\end{equation}
To obtain the second one, remark that $\Tor_1^R(G^n,k)\simeq G^n$ and
$\Tor_1^R(\mathcal{E}^{n},k)=0$ and apply $\cdot\otimes_R k$ to the exact triple $0\to\mathcal{E}^{n+1}\to \mathcal{E}^{n}\to G^n\to 0$.

By construction,
$(\mathcal{E}^{n+1}_K,\alpha^{n+1}_K)=({\mathcal{\mathcal{E}}}_K,{\alpha}_K)$
and $(\mathcal{E}^{n+1}_k,\alpha^{n+1}_k)$ is not $\mu$-semistable.
Let $(B^{n+1},\alpha_{B^{n+1}})$ be the maximal $\mu$-destabilizer of
$(\mathcal{E}^{n+1}_k,\alpha^{n+1}_k)$.
Let $C^n:=G^n\cap B^{n+1}$ and let
$\alpha_{C^n}:C^n\hookrightarrow G^n\overset{\alpha_{G^n}}\to\mathcal{F}_k$
be the induced framing.
Consider the two possible cases: (i) $\rk(B^n)>0$ and (ii) $\rk(B^n)=0$.

(i) $\rk(B^n)>0$.
One shows that
\begin{equation}\label{rk Bn+1}
\rk(B^{n+1})>0.
\end{equation}
Indeed, suppose the contrary, that is, assume  $B^{n+1}$ is a torsion sheaf. Then
\begin{equation}\label{deg Bn+1}
\deg(B^{n+1},\alpha_{B^{n+1}})>0.
\end{equation}
As $C^n$ is either zero or a torsion subsheaf of $G^n=\mathcal{E}^n_k/B^n$, we have
\begin{equation}\label{deg Cn}
\deg(C^n,\alpha_{C^n})\le0.
\end{equation}
On the other hand, the second triple in (\ref{2 triples}) shows that $B^{n+1}/C^n\hookrightarrow B^n$, hence by the
$\mu$-semistability of $(B^n,\alpha_{B^n})$, we have $\deg(B^{n+1}/C^n,{\alpha^n}')\le0$, where ${\alpha^n}'$ is the induced
framing. This
inequality, together with (\ref{deg Cn}), contradicts  (\ref{deg Bn+1}), which proves (\ref{rk Bn+1}). We thus may
assume that $(B^{n+1},\alpha_{B^{n+1}})$ is a $\mu$-semistable framed-saturated subsheaf of
$(\mathcal{E}^{n+1}_k,\alpha^{n+1}_k)$.
From (\ref{2 triples})--(\ref{rk Bn+1}) it follows that
$\mu(\mathcal{E}^{n+1}_k,\alpha^{n+1}_k)=\mu(\mathcal{E}^n_k,\alpha^n_k)=...=
\mu({\mathcal{E}}_k,{\alpha}_k)$.

Assume now that $\rk(B^{n+1}/C^n)=0$. Then $\rk C^n=\rk B^{n+1}>0$ and the inclusion $C^n\hookrightarrow G^n$ implies
$\mu(C^n,\alpha_{C^n})\le\mu_{\max}(G^n,\alpha_{G^n})<\mu(\mathcal{E}^n_k,\alpha^n_k)\le\mu(B^{n+1},\alpha_{B^{n+1}})$.
Hence $\deg(B^{n+1}/C^n,{\alpha^n}')>0$, contrary to the fact that $(B^{n+1}/C^n,{\alpha^n}')$ is a torsion subsheaf of
the $\mu$-semistable sheaf  $(B^n,\alpha_{B^n})$. Hence $\rk(B^{n+1}/C^n)>0$.
Therefore, since both $(B^n,\alpha_{B^n})$ and $(B^{n+1},\alpha_{B^{n+1}})$ are
$\mu$-semistable and $B^{n+1}/C^n\hookrightarrow B^n$, it follows that
$$
\mu(B^n,\alpha_{B^n})\ge\mu(B^{n+1},\alpha_{B^{n+1}}).
$$
In particular, $\beta_n=\mu(B^n,\alpha_{B^n})-\mu({\mathcal{E}}_k,{\alpha}_k)=
\mu(B^n,\alpha_{B^n})-\mu(\mathcal{E}^n_k,\alpha^n_k)$ is a positive rational number.
As $\{\beta_n\}_{n\ge1}$ is a descending sequence of strictly positive numbers in the lattice
$\frac{1}{r!}\mathbb{Z}\subset\mathbb{Q}$, where $r=\rk\mathcal{E}_k$, it is stationary.
We may assume without loss of generality that $\beta_n$ is constant for all $n$.
Then we have $C^n:=G^n\cap B^{n+1}=0$ in $\text{Coh}_{2,1}(X)$ for all $n$, where $\text{Coh}_{2,1}(X)$ is
the category of coherent sheaves on $X$ modulo sheaves of dimension 0. In particular, there are inclusions $B^{n+1}\subset B^n$ and $G^n\subset G^{n+1}$  in $\text{Coh}_{2,1}(X)$.  Hence there is an integer $n_0$
such that for all $n\ge n_0$ we have $\mu(B^n,\alpha_{B^n})=\mu(B^{n+1},\alpha_{B^{n+1}})=...,\
\rk{B^n}=\rk{B^{n+1}}=...$, and
we may assume that $n_0=0$.
In view of (\ref{2 triples})
\begin{equation}\label{Gn}
G^0\subset G^1\subset...
\end{equation}
is an increasing sequence of purely 2-dimensional sheaves such that
\begin{equation}\label{mu Gn}
\mu(G^0,\alpha_{G^0})=\mu(G^1,\alpha_{G^1})=...,\ \rk G^0=\rk G^1=...,\ \varepsilon(\alpha_{G^0})=
\varepsilon(\alpha_{G^1})=...
\end{equation}
This means that the Hilbert polynomials of the sheaves $G^i,i\ge0,$ coincide modulo constant terms. Equivalently, these
sheaves are isomorphic in dimension $\ge1$. In particular, their reflexive hulls $(G^n)^{\dual\dual}$ are
all isomorphic. Therefore, we may consider $\{G^n\}_{n\ge1}$ as a sequence of subsheaves in some
fixed coherent sheaf. As an immediate consequence, we obtain that all the injections are eventually
isomorphisms. Thus we may assume that $G^n\cong G^{n+1}$ for all $n\ge0$. This implies that the
short exact sequences (\ref{2 triples}) split, and we have $B^n = B,\ G^n = G$ and
$\mathcal{E}^n_k=B\oplus G$ for all $n$. Define $Q^n=\mathcal{E}/\mathcal{E}^n,\ n\ge0$. Then $Q^n_k=G$
and there are short exact sequences $0\to G\to Q^{n+1}\to Q^n\to0$ for all $n$. It follows from the local flatness
criterion \cite[Lemma 2.1.3]{HLbook} that $Q^n$ is an $R/\pi^n$-flat quotient of $\mathcal{E}/\pi^n\mathcal{E}$
for each $n$.  Hence the
image of the proper morphism $\sigma:\text{Quot}_{X_R/R}(\mathcal{E},P(G))\to\Spec(R)$ contains the closed subscheme $\Spec(R/(\pi)^n)$ for all $n$.  But this is only possible if $\sigma$ is surjective, so that $\mathcal{E}_{K'}$
must also admit a ($\mu$-destabilizing!) quotient with Hilbert polynomial $P(G(m))=\chi(G(m))$ for some field extension
$K'\supset K$. This contradicts the assumption on $\mathcal{E}_K$ .

(ii) $\rk(B^n)=0$. By (i), we obtain
$\rk(B^0)=\rk(B^1)=...=\rk(B^n)=0$. As $B^0$ is $\mu$-semistable,
it follows that $T(\ker\alpha_{B^0})=0$, i.e., $\alpha_{B^0}$ is injective. Then by the definition of $G^0$ we have
$\alpha_{G^0}=0$, i.e., $G^0$ has an ordinary $\mu$-Harder-Narasimhan filtration with $\mu$-semistable factors of
positive ranks. As $C^0$ is a torsion subsheaf of $G^0$, it follows that $C^0=0$. Hence, $B^1\hookrightarrow B^0$.
Repeating this argument we eventually obtain $B^{n+1}\hookrightarrow B^n\hookrightarrow...\hookrightarrow B^0$ and $\rk(B^{i})=0$ for all $i=0,\ldots,i+1$.

Again we obtain a decreasing sequence
$...\hookrightarrow B^n\hookrightarrow...\hookrightarrow B^0$ of subsheaves of the torsion sheaf $B^0$, which stabilizes in
$\text{Coh}_{2,1}(X)$ and corresponds to an increasing sequence (\ref{Gn}) of purely 2-dimensional sheaves
$G^i$
satisfying (\ref{mu Gn}). We conclude by the same argument as in~(i).
\qed

By using Proposition \ref{MW proj} and Theorem 3.3, and proceeding as in \cite{HLbook}, Proposition
8.2.6, we can prove the  following result.

\begin{prop}\label{fin gen ring}
There is an integer $N>0$ such that $\bigoplus_{l\ge0}W_{lN}$ is a finitely generated graded ring.
\end{prop}
We can eventually  define the Uhlenbeck-Donaldson compactification.
\begin{defin}\label{def M-muss}
Let $N$ be a positive integer as in the above proposition. Then
$M^{\mu ss}=M^{\mu ss}(c,\delta)$
is defined by
$$
M^{\mu ss}={\rm Proj}\left(\bigoplus_{k\ge0}H^0(R^{\mu ss}(c,\delta),
\mathcal{L}(n_1,n_2)^{kN})^{SL(P(m))}\right)\ .
$$
It is equipped with a natural morphism $\pi:R^{\mu ss}(c,\delta)\to M^{\mu ss}$
and is called the moduli space of $\mu$-semistable framed
sheaves.
\end{defin}

We introduce more notation. Let
$\mathbf{F}=(\mathbf{E},\mathbf{L},\alpha_{\mathbf{E}})\in\mathcal{M}^{\mu ss}(S)$ be a family of framed sheaves.
Consider the scheme
$\widetilde{S}:=\mathrm{\mathbb{I}som}(V\otimes\mathcal{O}_S,\mathrm{pr}_{1*}\mathbf{E})\overset{\tau}\to S$
together with the projections $\widetilde{S}\overset{\widetilde{\mathrm{pr}}_1}\leftarrow
\widetilde{S}\times X\overset{\widetilde{\mathrm{pr}}_2}\to X$.
Let $\widetilde{\mathbf{F}}=(\widetilde{\mathbf{E}},\widetilde{\mathbf{L}},\alpha_{\widetilde{\mathbf{E}}})
\in\mathcal{M}^{\mu ss}(\widetilde{S})$ be the lifted family over $\widetilde{S}$, where
$\widetilde{\mathbf{E}}:=(\tau\times\mathrm{id}_X)^*\mathbf{E},\widetilde{\mathbf{L}}:=\tau^*\mathbf{L}$,
$\alpha_{\widetilde{\mathbf{E}}}:=\tau^*\alpha_{\mathbf{E}}:\widetilde{\mathbf{L}}\to
\widetilde{\mathrm{pr}}_{1*}\mathcal{H}om(\widetilde{\mathbf{E}},\mathcal{O}_{\widetilde{S}}\boxtimes\mathcal{F})$.
Let $\mathrm{taut}:V\otimes\mathcal{O}_{\widetilde{S}}\overset{\sim}\to\tau^*\mathrm{pr}_{1*}\mathbf{E}=
\widetilde{\mathrm{pr}}_{1*}\widetilde{\mathbf{E}}$
be the tautological isomorphism.
Applying the functor $\mathrm{pr}_{1*}$ to the morphism
$\widetilde{\alpha}_{\mathbf{E}}:\mathrm{pr}_1^*\mathbf{L}\otimes\mathbf{E}\to\mathcal{O}_S\boxtimes\mathcal{F}$
we obtain a (nowhere vanishing) morphism
$\widehat{\alpha}_{\mathbf{E}}:\mathbf{L}\otimes\mathrm{pr}_{1*}\mathbf{E}\to H^0(\mathcal{F})\otimes\mathcal{O}_S$.
Consider the composition
$a_{\widetilde{\mathbf{E}}}:V\otimes\mathcal{O}_{\widetilde{S}}\otimes\widetilde{\mathbf{L}}\overset{\mathrm{taut}}\to
\widetilde{\mathrm{pr}}_{1*}\widetilde{\mathbf{E}}\otimes
\widetilde{\mathbf{L}}\overset{\widehat{\alpha}_{\widetilde{\mathbf{E}}}}\to
H^0(\mathcal{F})\otimes\mathcal{O}_{\widetilde{S}}$
or, equivalently, the (subbundle) morphism
$a_{\widetilde{\mathbf{E}}}:\widetilde{\mathbf{L}}\to\mathcal{H}om(V\otimes\mathcal{O}_{\widetilde{S}},
H^0(\mathcal{F})\otimes\mathcal{O}_{\widetilde{S}})\cong
\mathrm{Hom}(V,H^0(\mathcal{F}))\otimes\mathcal{O}_{\widetilde{S}}$.
By the universal property of the projective space $\mathbb{P}$ the subbundle morphism $a_{\widetilde{\mathbf{E}}}$
defines a morphism $b_{\widetilde{\mathbf{E}}}:\widetilde{S}\to\mathbb{P}$.

Now we   explain in which sense $M^{\mu ss}$ is the moduli space of $\mu$-semistable framed sheaves.
In fact, though $M^{\mu ss}$ is not in general a categorial quotient of $R^{\mu ss}$,
still $M^{\mu ss}$ has the following universal property. Let $\mathcal{M}^{ss}$, respectively,
$\widetilde{\mathcal{M}}^{\mu ss}$ denote the functor which associates with $S$ the set of isomorphism
classes of $S$-flat families of semistable, respectively,   $\mu$-semistable framed sheaves of class $c$ on $X$.
Consider an open subfunctor $\mathcal{M}^{\mu ss}$ of $\widetilde{\mathcal{M}}^{\mu ss}$ which associates with $S$
the set of isomorphism classes of those families
$[\mathbf{F}]\in\widetilde{\mathcal{M}}^{\mu ss}(S)$
for which there exists a dense open subset $S'$ of $S$ such that
$[\mathbf{F}|_{S'\times X}]\in\mathcal{M}^{ss}(S')$.
Clearly, $\mathcal{M}^{ss}$ is an open subfunctor $\mathcal{M}^{\mu ss}$.

For any scheme $S$ and any family $[\mathbf{F}=(\mathbf{E},\alpha_\mathbf{E})]\in\mathcal{M}^{\mu ss}(S)$
the principal $GL(V)$-bundle $\tau:\widetilde{S}\to S$ by the universality of the Quot-scheme
$\text{Quot}(\mathcal{H},P_c)$ defines a morphism
$\Psi_{\widetilde{\mathbf{F}}}:\widetilde{S}\to\text{Quot}(\mathcal{H},P_c)$
and hence a morphism
$\Phi^{\mu}_{\widetilde{\mathbf{F}}}=(\Psi_{\widetilde{\mathbf{F}}},
b_{\widetilde{\mathbf{E}}}):\widetilde{S}\to R^{\mu ss}$.
This morphism is
$GL(V)$-invariant, and $\tau:\widetilde{S}\to S$ is a categorial quotient, so that there exists a (classifying)
morphism $\Phi_{\mathbf{F}}:S\to M^{\mu ss}$ making the following diagram commutative:
\begin{equation}\label{morphism tau}
\xymatrix{\widetilde{S}\ar[r]^-
{\Phi^{\mu}_{\widetilde{\mathbf{F}}}}\ar[d]_{\tau} & R^{\mu ss}\ar[d]^{\pi} \\
S\ar[r]^-{\Phi^{\mu}_{\mathbf{F}}} & M^{\mu ss}.}
\end{equation}

We thus obtain a natural
transformation of functors $\Phi^{\mu}:\mathcal{M}^{\mu ss}\to{\rm Mor}(-,M^{\mu ss})$
given by
$\Phi^{\mu}(S):\mathcal{M}^{\mu ss}(S)\to{\rm Mor}(S,M^{\mu ss}),\ [\mathbf{F}]\mapsto\Phi^{\mu}_{\mathbf{F}}$.

Let $M=M(c,\mathcal{F})$ denote the moduli space of semistable framed sheaves
$(E,\alpha:E\to\mathcal{F})$ on $X$ with $\ch (E)=c$.
It co-represents the
moduli functor $\mathcal{M}^{ss}=\mathcal{M}^{ss}(c,\mathcal{F})$,
namely, we have a natural transformation of functors
$\Phi:\mathcal{M}^{ss}\to {\rm Mor}(-,M)$,
$\Phi(S):\mathcal{M}(S)\to{\rm Mor}(S,M),\ [\mathbf{F}]\mapsto\Phi_{\mathbf{F}}$
\cite[Thm 0.1]{HL1}, and the above diagram extends to a commutative diagram
$$
\xymatrix{\widetilde{S}\ar@/^2pc/[rrr]^-{\Phi^{\mu}_{\widetilde{\mathbf{F}}}}\ar[rr]^-{\Phi_{\widetilde{\mathbf{F}}}}
\ar[d]_{\tau} & &
R^{ss}\ar@{^{(}->}[r]\ar[d]^{\pi} & R^{\mu ss}\ar[d]^{\pi} \\
S\ar@/_2pc/[rrr]^{\Phi^{\mu}_{\mathbf{F}}} \ar[rr]^-{\Phi_{\mathbf{F}}}
& & M\ar@{-->}[r]^-\gamma & M^{\mu ss}.}
$$
Since $\pi:R^{ss}\to M$ is a categorial quotient, it follows that there exists a morphism
$\gamma:M\to M^{\mu ss}$ such that $\Phi^{\mu}=\underline{\gamma}\cdot\Phi$, i.e.,
$\Phi^{\mu}_{\mathbf{F}}=\gamma\cdot\Phi_{\mathbf{F}}$. The morphism $\gamma$ is by construction
dominant and projective, hence it is surjective. It is also birational on the components of $M$ containing at least one locally free framed sheaf.

Note also that, for any $S$ and any $[\mathbf{F}=(\mathbf{E},\alpha_\mathbf{E})]\in\mathcal{M}^{\mu ss}(S)$
the pullback of
$\mathcal{O}_{M^{\mu ss}}(1)$ via
$\Phi^{\mu}_{\mathbf{F}}\cdot\tau$ is isomorphic to
$(\lambda_{\widetilde{\mathbf{E}}}(u_1(c))^{\otimes n_1}\otimes
b_{\widetilde{\mathbf{E}}}^*\mathcal{O}_{\mathbb{P}}(n_2))^N$. In particular, if
$[\mathbf{F}]\in\mathcal{M}^{ss}(S)$, then
\begin{equation}\label{O(1)}
(\gamma\cdot\Phi_{\mathbf{F}}\cdot\tau)^*\mathcal{O}_{M^{\mu ss}}(1)\cong
(\lambda_{\widetilde{\mathbf{E}}}(u_1(c))^{\otimes n_1}
\otimes b_{\widetilde{\mathbf{E}}}^*\mathcal{O}_{\mathbb{P}}(n_2))^N.
\end{equation}
We thus obtain:
\begin{thm}\label{gamma}
The morphism of functors
$\mathcal{M}^{ss}\to\mathcal{M}^{\mu ss}$ induces a regular
morphism of moduli spaces $\gamma:M\to M^{\mu ss}$ such
that (\ref{O(1)}) is satisfied for any $S$ and any $[\mathbf{F}]\in\mathcal{M}^{ss}(S)$.
Moreover, $\gamma$ is birational on the components of $M$ that contain at least one locally free framed sheaf.
\end{thm}

Let now $M^{\mu\mbox{\scriptsize -}\mathrm{stable}}$, $M^{\mu\mbox{\scriptsize -}\mathrm{poly}}$
be the open subsets of $M$ corresponding to $\mu$-stable, resp. $\mu$-polystable pairs $(E,\alpha)$  with $E$ locally free.
We are assuming that $M^{\mu\mbox{\scriptsize -}\mathrm{stable}}$ is nonempty.
We shall see (Theorem \ref{idinst}) that the restriction $M^{\mu\mbox{\scriptsize -}\mathrm{poly}}\xrightarrow{\ \gamma \ }M^{\mu ss}$
is injective. Actually, when restricted to $M^{\mu\mbox{\scriptsize -}\mathrm{stable}}$, this map is an embedding, so that by taking the closure  of $\gamma(M^{\mu\mbox{\scriptsize -}\mathrm{stable}})$ in
$M^{\mu ss}$, we obtain a compactification of $M^{\mu\mbox{\scriptsize -}\mathrm{stable}}$.
By analogy with the  nonframed
case, we will call it the {\em Uhlenbeck-Donaldson compactification} of $M^{\mu\mbox{\scriptsize -}\mathrm{stable}}$.

With  reference to the notation introduced in the beginning
of Section \ref{family}, we set
\begin{multline*}
\mathcal{S}^{\mu ss}(c,\delta)^*:=\{(E,\alpha)\in\mathcal{S}^{\mu ss}(c,\delta)\ |\ E\
\text{is locally free at all points of $D$}\\
\mbox{and $\alpha$ induces an isomorphism $E|_D\simeq \F$}\},
\end{multline*}
$$
R^{\mu ss}(c,\delta)^*:=\{([g:\mathcal{H}\to E],[\alpha\circ g])\in R^{\mu ss}(c,\delta)\ | (E,\alpha)\in\mathcal{S}^{\mu ss}(c,\delta)^\ast\},$$
$$
M^{\mu ss}(c,\delta)^*:=\pi(R^{\mu ss}(c,\delta)^*)\ ,
\ \ \ \ \  M^*:=\gamma^{-1}(M^{\mu ss}(c,\delta)^*).
$$
Note that the starred versions of
$\mathcal{S}^{\mu ss}(c,\delta)$,
$ R^{\mu ss}(c,\delta)$,
and $M$ are open in the respective non-starred ones, and $M^{\mu ss}(c,\delta)^*$ is a priori only constructible in $M^{\mu ss}(c,\delta)$. Obviously, $M^{\mu\mbox{\scriptsize -}\mathrm{poly}}\subset M^*$.

We now proceed to a more detailed study of the fibers of the restriction
$
\gamma:M^*\to M^{\mu ss}(c,\delta)
$
over the points of its image $\gamma (M^*)\subset M^{\mu ss}(c,\delta)^*$.

\subsection{Description of $\boldsymbol{\gamma|_{M^*}}$}
\label{descr_mor}
Let $(E,\alpha)\in\mathcal{S}^{\mu ss}(c,\delta)^*$. Consider the {\it  graded
framed sheaf} $gr^{\mu}(E,\alpha)=(gr^{\mu}E,gr^{\mu}\alpha)$ associated with some
$\mu$-Jordan-H\"older filtration of $(E,\alpha)$. It is $\mu$-polystable
as a framed sheaf. Remark that, applying
the definition of $\mu$-semistability to $E(-D)= \ker\alpha\subset E$, one concludes
that $\delta_1\leq r\deg D$.
Moreover, in the case of equality, $(E(-D), 0)\subset (E,\alpha)$ is the
upper level of the Jordan--H\"older filtration with torsion quotient.
Under our hypotheses, this is the only possible torsion in the graded object
associated with the Jordan--H\"older filtration. To eliminate it, we impose,
from now on, the additional hypothesis $\delta_1<r\deg D$.

By taking the double dual we get a $\mu$-polystable locally-free
framed sheaf  $(gr^{\mu}E)^{\dual\dual}$.
The function
$l_E:X\to\mathbb{N}\cup\{0\}:x\mapsto\operatorname{length}\big((gr^{\mu}E)^{\dual\dual}
/gr^{\mu}E\big)_x$ can be
considered as an element in the symmetric product $S^l(X\setminus D)$ with
$l=c_2(E)-c_2((gr^{\mu}E)^{\dual\dual})$.
Both $(gr^{\mu}E)^{\dual\dual}$ and $l_E$ are well-defined invariants of $(E,\alpha)$, i.e., they do not
depend on the choice of a $\mu$-Jordan-H\"older filtration of $(E,\alpha)$.

\begin{thm}\label{idinst}
Assume that $\delta_1<r\deg D$.
Two framed sheaves
$(E_1,\alpha_1),\ (E_2,\alpha_2)$ from $\mathcal{S}^{\mu ss}(c,\delta)^*\cap \mathcal{S}^{ss}(c,\delta)$
define the same closed  point in $M^{\mu ss}(c,\delta)^*$ if and only if
$$
(gr^{\mu}(E_1,\alpha_1))^{\dual\dual}=(gr^{\mu}(E_2,\alpha_2))^{\dual\dual}\ \ \ and\ \ \ l_{E_1}=l_{E_2}.
$$
\end{thm}
\begin{proof}
The proof goes along the same lines as that of \cite[Theorem 8.2.11]{HLbook}.
We start with the ``if'' part. Take any framed
sheaf $(E,\alpha)$ whose $S$-equivalence class belongs to $M^*$, that is $(E,\alpha)\in\mathcal{S}^{\mu ss}(c,\delta)^*\cap \mathcal{S}^{ss}(c,\delta)$, and consider the graded framed sheaf
$gr^{\mu}(E,\alpha)$ obtained from some $\mu$-Jordan-H\"older filtration of $(E,\alpha)$. Then
one can naturally construct a flat family $(\mathbf{E},\mathbf{A})$ of framed sheaves over
$\mathbb{A}^1$ such that

 i) $(E_t,\alpha_t)\cong(E,\alpha)$ for all
$0\ne t\in\mathbb{A}^1$ , and

ii) $(E_0,\alpha_0)\cong gr^{\mu}(E,\alpha)$.

\noindent The classifying morphism
$\Phi_{\mu}:\mathbb{A}^1\to M^{\mu ss}(c,\delta)^*$ factors into the composition
$\Phi_{\mu}:\mathbb{A}^1\overset{\Phi}\to M^*\overset{\gamma}\to M^{\mu ss}(c,\delta)^*$,
where $\Phi$ is the classifying morphism. By i) $\Phi(\mathbb{A}^1)$ is a point, hence also
$[(E,\alpha)]:=\Phi_{\mu}(E,\alpha)=\Phi_{\mu}(\mathbb{A}^1)$ is a point, and by ii) we have
$[(E,\alpha)]=[gr^{\mu}(E,\alpha)]$. It follows that it is enough to consider $\mu$-polystable
framed sheaves from $\mathcal{S}^{\mu ss}(c,\delta)^*$.

Thus, let $(E,\alpha)$ be a $\mu$-polystable framed sheaf from $\mathcal{S}^{\mu ss}(c,\delta)^*$.
Then $\mathcal{E}:=E^{\dual\dual}$ is $\mu$-polystable and locally free, and
there is an exact sequence
$$
0\to E\stackrel{can}{\relbar\joinrel\relbar\joinrel\rightarrow}\mathcal{E}\stackrel{\epsilon}{\relbar\joinrel\rightarrow} T\to0
$$
where
$T$ is a torsion sheaf with $l(T)=l_E$.
Furthermore, $E$ is locally free
along the framing curve $D$ by the definition of $\mathcal{S}^{\mu ss}(c,\delta)^*$,
hence there exists a morphism
$\alpha_D:\mathcal{E}|_D\to\mathcal{F}$
such that the framing $\alpha:E\to\mathcal{F}$ decomposes as
$$
\alpha\colon E \stackrel{\otimes\mathcal{O}_D}{\relbar\joinrel\relbar\joinrel\relbar\joinrel\rightarrow} E|_D
\cong \mathcal{E}|_D\stackrel{\alpha_D}{\relbar\joinrel\relbar\joinrel\rightarrow} \mathcal{F}.
$$
Consider the morphism
$\psi:{\rm Quot}(\mathcal{E},l)\to S^lX:
[\mathcal{E}\overset{\epsilon}\twoheadrightarrow T]\mapsto l_{E_\epsilon}$, where
$E_\epsilon:=\ker\epsilon$, and set
$$
{Y_E}:=\psi^{-1}(l_E).
$$
There is a universal exact triple
$$
0\to\mathbb{E}\to\mathcal{O}_{Y_E}\boxtimes\mathcal{E}\to\mathbb{T}\to0
$$
of families on $X$ parametrized by ${Y_E}$, where $\mathbb{T}$ is the family of artinian sheaves of length $l$ on $X$.
Let $p_1:{Y_E}\times X\to {Y_E}$ be the projection onto the first factor and set
$\widetilde{Y}_E:={\rm \mathbb{I}som}
(\mathcal{O}_{Y_E}\otimes V,p_{1*}(\mathcal{O}_{Y_E}\boxtimes\mathcal{E}(m)))
\overset{p_E}\to {Y_E}$
and
$\mathbb{E}_{\widetilde{Y}_E}:=(p_E\times id_X)^*\mathbb{E}$. Note that ${\widetilde{Y}_E}$ is a trivial
$GL(V)$-bundle on ${Y_E}$.
For any $w\in {\widetilde{Y}_E}$ we have a tautological epimorphism
$g_w:\mathcal{H}\to E_w:=\mathbb{E}_{\widetilde{Y}_E}|\{w\}\times X$.
By the universal property of $R^{\mu ss}(c,\delta)^*$ there is a well defined morphism
\begin{eqnarray*}
\Phi_{\widetilde{Y}_E} \colon  {\widetilde{Y}_E} & \to &  R^{\mu ss}(c,\delta)^*  \\ w &\mapsto &
(g_w, z)\ ,\ \ \mbox{where}\\
z &=&[\mathcal{H}\overset{g_w}\to E_w\overset{\otimes
\mathcal{O}_D}\longrightarrow E_w|_D\cong
\mathcal{E}|_D\overset{\alpha_D}\longrightarrow\mathcal{F}]\ ,
\end{eqnarray*}
and according to (\ref{morphism tau}) we have a commutative diagram
$$
\xymatrix{
{\widetilde{Y}_E}\ar[d]^{p_E} \ar[rr]^-{\Phi_{\widetilde{Y}_E}} &
& R^{\mu ss}(c,\delta)^*\ar[d]^-{\pi} \\
{Y_E}\ar[rr]^-{\Phi_{Y_E}} && M^{\mu ss}(c,\delta)^*,}
$$
where
\begin{equation}\label{Phi{Y_E}}\begin{array}{c}
\Phi_{Y_E}:\ {Y_E} \to  M^{\mu ss}(c,\delta)^*\,,\\[5pt]  y \mapsto [(E_y=\mathbb{E}|\{y\}\times X,\ \alpha_y:E_y\overset{\otimes\mathcal{O}_D}\longrightarrow E_y|_D\cong
\mathcal{E}|_D\overset{\alpha_D}\longrightarrow\mathcal{F})]
\end{array} \end{equation}
is the classifying morphism. From this diagram and formula (\ref{O(1)}) it follows that
\begin{equation}\label{again O(1)}
(\Phi_{Y_E}\circ p_E)^*\mathcal{O}_{M^{\mu ss}(c,\delta)^*}(1)\cong
(\lambda_{\mathbb{E}}(u_1)^{\otimes n_1}\otimes\operatorname{pr}^\ast\cO_{\PP}(n_2))^N,
\end{equation}
where $\operatorname{pr}:R^{\mu ss}(c,\delta)^*\to\mathbb{P}$ is the projection.
One shows that the right hand side of (\ref{again O(1)}) is trivial. In fact, since $\psi({Y_E})=l_E$ is a
point, it follows from the computations in \cite[Example 8.2.1]{HLbook} that
$\lambda_{\mathbb{E}}(u_1)=\mathcal{O}_{Y_E}$, hence
$\lambda_{\mathbb{E}_{\widetilde{Y}_E}}(u_1)=\mathcal{O}_{\widetilde{Y}_E}.$
On the other hand, the above diagram shows that
$\Phi_{\widetilde{Y}_E}^*\operatorname{pr}^\ast\cO_{\PP}(1)=\mathcal{O}_{\widetilde{Y}_E}.$
Whence (\ref{again O(1)}) yields
\begin{equation}\label{trivial}
(\Phi_{Y_E}\circ p_E)^*\mathcal{O}_{M^{\mu ss}(c,\delta)^*}(1)\cong\mathcal{O}_{\widetilde{Y}_E}.
\end{equation}
Note that ${Y_E}$ is irreducible (see, e.g., \cite{EL}) and $p_w:{\widetilde{Y}_E}\to {Y_E}$ is a trivial
principal bundle; hence ${\widetilde{Y}_E}$ is also an irreducible scheme. It follows now from
(\ref{trivial}) that $y=\Phi_{Y_E}({Y_E})$ is a point.
In particular, (\ref{Phi{Y_E}}) shows that $y=[(E,\alpha_E)]=[(E,\alpha_{E'})]$ which
proves the ``if''  part of the theorem.

The proof of the  ``only if'' part uses the restriction Theorem \ref{Meh-Ram2} and
 will require the  next Lemma and Proposition.

\begin{lemma}\label{HLbook,Lem.8.2.12}
Let $F_i=(E_i,\alpha_{E_i}),\ i=1,2,$ be framed $\mu$-semistable sheaves on $X$ such that $E_1$ and $E_2$ are locally
free along $D$. Let $a$ be a sufficiently large integer and $C\in|aH|$ a general smooth curve. Then
$F_1|_C$ and $F_2|_C$ are $S$-equivalent if and only if $(gr^\mu F_1)^{\dual\dual}=(gr^\mu F_2)^{\dual\dual}$.
\end{lemma}

\begin{proof}
Let $gr^\mu F_1=\bigoplus\limits_{i=1}^n(E_i/E_{i-1},\alpha_i)$ be the graded object of a
$\mu$-Jordan-H\"older filtration of $F_1$. According to Theorem \ref{Meh-Ram} one can choose $a$ large enough so that
the restriction of any summand $E_i/E_{i-1}$ is $\mu$-stable again. Now choose a $C$   that   avoids the
finite set of all singular points of the sheaves $E_i/E_{i-1}$ for all $i$. Then
$gr^\mu F_1|_C=(gr^\mu F_1)^{\dual\dual}|_C$ is the graded object of a $\mu$-Jordan-H\"older filtration of $F_1|_C$.
In view of Remark \ref{S-eq}, this shows that for a general curve $C$ of sufficiently high degree, $F_1|_C$ and
$F_2|_C$ are $S$-equivalent if $(gr^\mu F_1)^{\dual\dual}|_C\cong(gr^\mu F_2)^{\dual\dual}|_C$.
For $a>>0$ and $i=0,1$ we have
$$\Ext^i((gr^\mu F_1)^{\dual\dual},(gr^\mu F_2)^{\dual\dual}(-C))=
\Ext^i((gr^\mu F_2)^{\dual\dual},(gr^\mu F_1)^{\dual\dual}(-C))=0\,,$$  so that
$$\Hom((gr^\mu F_1)^{\dual\dual},(gr^\mu F_2)^{\dual\dual})
\cong\Hom((gr^\mu F_1)^{\dual\dual}|_C,(gr^\mu F_2)^{\dual\dual}|_C)\,.$$
This means that $(gr^\mu F_1)^{\dual\dual}|_C$ $\cong(gr^\mu F_2)^{\dual\dual}|_C$ if and only if
$(gr^\mu F_1)^{\dual\dual}\cong(gr^\mu F_2)^{\dual\dual}$.
\end{proof}

From this Lemma and the second claim in Lemma \ref{separate} it follows that if
$(gr^\mu F_1)^{\dual\dual}\not\cong(gr^\mu F_2)^{\dual\dual}$
then any two points $y_1$ and $y_2$ in $R^{\mu ss}(c,\delta)$ representing $F_1$ and $F_2$
are separated by $\operatorname{SL}(V)$-invariant sections of
$\mathcal{L}(n_1,n_2)^{\otimes\nu Nk}$ for some $\nu>0$, where
$\mathcal{L}(n_1,n_2)^{\otimes\nu Nk}=(\gamma\circ\pi)^*\mathcal{O}_{M^{\mu ss}}(\nu)$
(see Definition \ref{def M-muss}). This means that $\gamma(y_1)\ne\gamma(y_2)$. We thus consider the case
$(gr^\mu F_1)^{\dual\dual}\cong(gr^\mu F_2)^{\dual\dual}=:\mathcal{E}$ but $l_{F_1}\ne l_{F_2}$.

As we have seen, $\gamma$ is constant on the fibres of the morphism
$\psi:{\rm Quot}(\mathcal{E},l)\to S^lX, [\mathcal{E}\overset{\epsilon}\twoheadrightarrow T]\mapsto l_{E_\epsilon}$.
As $\psi$ is surjective and $S^lX$ is normal, $\gamma|_{{\rm Quot}(\mathcal{E},\ l)}$ factors through a morphism
$j:S^lX\to M^{\mu ss}$. The proof of Theorem is complete if we can show the following proposition.

\begin{prop}\label{HLbook,Prop.8.2.13}
The morphism $j: S^lX\to M^{\mu ss}$ is a closed immersion.
\end{prop}
\begin{proof}
It is well known that, for a smooth curve $C\in|aH|$, the subset $$\{Z\in S^lX\ |\ {\rm Supp}Z\cap C\ne\emptyset\}$$
of $S^lX$ has a structure of an ample irreducible reduced Cartier divisor which we will denote by $\widetilde{C}$.
Consider the above quoted morphism $\psi:{\rm Quot}(\mathcal{E},l)\to S^lX$ associated with the family $\mathbb{T}$.
Apply the argument from the proof of Lemma 8.2.15 in
\cite{HLbook} to formula (\ref{mono}) with $S={\rm Quot}(\mathcal{E},l)$ . Since $\mathcal{E}|_C$ is $\mu$-polystable,
it follows that there exists an integer $\nu>0$ and a
section $\sigma\in H^0(Y_C,\mathcal{L}'_0(n_1,n_2k)^{\otimes\nu})^{\operatorname{SL}(V_C)}$ such that the zero divisor
of $s_{\mathcal{E}}(\sigma)$ is a multiple of $\psi^{-1}(\widetilde{C})$. This implies that
$\sigma$ induces a section $\sigma'$ of some tensor power of $\mathcal{O}_{M^{\mu ss}}(1)$
such that the zero scheme of the section $j^*(\sigma')$ is a multiple of $\widetilde{C}$. The divisors
$\widetilde{C}$ span a very ample linear system on $S^lX$ as $C$ runs through all smooth curves in the linear system
$|aH|$ for $a$ large enough. Hence $j$ is an embedding.
\end{proof}

This finishes the proof of Theorem \ref{idinst}.
\end{proof}

From this Theorem
we obtain a set-theoretic stratification of the Uhlenbeck-Donaldson compactification.
\begin{corol}\label{M*}
Let $c=(r,\xi,c_2)$ be a numerical K-theory class and let
$M^{\mu\mbox{\scriptsize -}\mathrm{poly}}(r,\xi,c_2,\delta)^\ast \subset M^{\mu ss}(c,\delta)^*$
denote the subset corresponding to $\mu$-polystable locally-free sheaves.
Assume, as before, that $\delta_1<r\deg D$. One has
the following set-theoretic stratification:
\begin{equation*}
M^{\mu ss}(c,\delta)^*=\underset{l\ge0}\coprod M^{\mu\mbox{\scriptsize -}\mathrm{poly}}(r,\xi,c_2-l,
\delta)^\ast\times S^l(X\setminus D).
\end{equation*}
\end{corol}

\begin{rema}
By our definition, $M^{\mu ss}=M^{\mu ss}(c,\delta)$, $M^{\mu ss*}=M^{\mu ss}(c,\delta)^*$ and $\gamma$ depend on the
choice of $m\ge m_0$, a positive integer used in the definition of the vector space $V$.
So it is more natural to denote them $M^{\mu ss}_m,\ M^{\mu ss*}_m$ and $\gamma_m$.
As shows Corollary \ref{M*}, $M^{\mu ss}(c,\delta)^*=M^{\mu ss*}_m$ does not depend on $m$ at least set-theoretically.
To obtain a Uhlenbeck-Donaldson type compactification $M^{\mu ss}$ of
$M^{\mu\mbox{\scriptsize -}\mathrm{poly}}(r,\xi,c_2,\delta)^\ast$, which is  a projective scheme
independent of $m$, one can proceed as follows.

Consider the sequence of morphisms $\{\gamma_m:M\to M_m^{\mu ss}\}_{m\ge m_0}$.
Define inductively a new series of morphisms $\{\gamma_{(k)}:M\to M_{(k)}\}_{k\ge0}$ as follows.
For $k=0$ set $M_{(0)}:=M^{\mu ss}_{m_0}$ and $\gamma_{(0)}:=\gamma_{m_0}:M\to M_{(0)}$.
Now, for $k\ge0$, assume that the scheme $M_{(k)}$ and a regular birational morphism $\gamma_{(k)}:M\to M_{(k)}$ are
already defined. Consider the morphism
$\gamma_{(k+1)}:=(\gamma_{(k)},\gamma_{m_0+k+1}):M\to M_{(k)}\times M_{m_0+k+1}^{\mu ss})$
and let $M_{(k+1)}$ be the scheme-theoretic image of the morphism $\gamma_{(k+1)}$ (in the usual sense of
\cite[II, Ex. 3.11(d)]{H}), together with a regular birational projection $\delta_{k+1}:M_{(k+1)}\to M_{(k)}$ such that
$\gamma_{(k)}=\delta_k\cdot\gamma_{(k+1)}$. We thus obtain for any $k\ge1$ a decomposition of the birational  morphism
$\gamma_{(0)}:M\to M_{(0)}$ into the composition
$$
\gamma_{(0)}=\delta_1\cdot...\cdot\delta_k\cdot\gamma_{(k)},\ \ \ k\ge1.
$$
As $\gamma_{(0)}$ is a birational projective morphism, it follows that there exists an integer $k_0$ such that
$\gamma_k=\gamma_{k_0}$ and $\delta_k=id$ for $k\ge k_0$. We now define the space $M^{\mu ss}$ and, respectively, the
morphism $\gamma:M\to M^{\mu ss}$ as
$$
M^{\mu ss}=M^{\mu ss}(c,\delta):=M_{(k_0)},\ \ \ \ \gamma:=\gamma_{k_0}:M\to M^{\mu ss}.
$$
These definitions do not depend on $m$.
\end{rema}

\bigskip\section{Concluding remarks}
Let $X$ be a smooth projective surface, and let $D$ be a big
and nef irreducible divisor in $X$.
Let  $E_D$ be a  locally-free sheaf on $D$ such that
there exists a real number $A_0$, $0\leq A_0< \frac1r D^2$ with the following property: for any
locally-free subsheaf $ F\subset  E_D$ of constant positive rank, one has $\frac{1}{\rk F}\deg c_1(F) \leq \frac{1}{\rk E_D}\deg c_1(E_D)+A_0$. Considering $E_D$ as a sheaf on $X$,
we say that a framed sheaf $(E,\alpha\colon E \to E_D)$ is $(D,E_D)$-framed
if $(E,\alpha)$ satisfies the condition of the definition of
$\mathcal{S}^{\mu ss}(c,\delta)^*$, that is $E$ is locally free along $D$
and $\alpha_{\vert D}$ is an isomorphism between $E_{\vert D}$ and $E_D$.
It was shown in \cite{DimaUgo} that for any $c\in H^*(X,\QQ)$ there exists an ample divisor $H$ on $X$ and
a real number $\delta>0$ such that all the $(D,E_D)$-framed sheaves
$\E$ on $X$ with Chern character $\ch(\E)=c$ are $(H,\delta)$-stable. As a consequence,
one has a moduli space for  $(D,E_D)$-framed sheaves on $X$, which embeds as an open
subset into the moduli space of stable pairs. These moduli spaces have been
quite extensively studied in connection with instanton counting and Nekrasov partition
functions (see \cite{NY,BPT,GaLiu}  among others).

 Let us in particular consider the open subset
formed by locally-free $(D,E_D)$-framed sheaves on $X$. By restricting
the previous construction to this open subset we construct a Uhlenbeck-Donaldson
partial compactification for it (we call this ``partial'' because the moduli
space of slope semistable framed bundles is not projective in general in this case). This generalizes the construction done by Nakajima, using ADHM data,  when $X$ is the complex projective plane.
An extension to a general projective surface
was hinted at in \cite{NY-lect} but was not carried out.

\end{document}